\documentclass[11pt,reqno]{amsart}
\usepackage{amsfonts}
\usepackage{amsthm, amsfonts, amssymb, color}
\usepackage{mathrsfs}
\usepackage{caption}
\usepackage{enumerate}
\usepackage{cases}
\usepackage{amsmath}
\usepackage{mathbbol}
\usepackage{footnote}
\usepackage{amstext, amsxtra}
\usepackage{txfonts}
\usepackage{tikz}

\usepackage[colorlinks, linkcolor=black, citecolor=blue,   hypertexnames=false]{hyperref}
\usepackage{graphicx}
\usepackage{subfigure}
\usepackage[symbol]{footmisc}
\allowdisplaybreaks
\def \and {{\qquad\text{and}\qquad}}

\newtheorem{theorem}{Theorem}[section]
\newtheorem{proposition}[theorem]{Proposition}
\newtheorem{corollary}[theorem]{Corollary}
\newtheorem{lemma}[theorem]{Lemma}
\newtheorem{definition}[theorem]{Definition}

\newtheorem{remark}[theorem]{Remark}

\numberwithin{equation}{section}
\theoremstyle{definition}

\title[UP and Geometric Condition for the Observability of Schr\"{o}dinger Equations]
{Uncertainty Principle and Geometric Condition for the Observability of Schr\"{o}dinger Equations}

\author{Longben Wei \quad Zhiwen Duan \quad Hui Xu }

\address{Longben Wei,  School of Mathematics and Statistics, Huazhong University of Science and Technology,  Wuhan,  430074,  P.R. China}
\email{d202080006@hust.edu.cn}
\address{Zhiwen Duan,  School of Mathematics and Statistics, Huazhong University of Science and Technology,  Wuhan,  430074,  P.R. China}
\email{duanzhw@hust.edu.cn}
\address{Hui Xu,  School of Mathematics and Statistics, Huazhong University of Science and Technology,  Wuhan,  430074,  P.R. China}
\email{d202180016@hust.edu.cn}


\begin{document}
	
	\begin{abstract}
	We provide necessary and sufficient geometric conditions for the exact observability of the Schrödinger equation with inverse-square potentials on the half-line. These conditions are derived from a Logvinenko-Sereda type theorem for generalized Fourier transform. Specifically, the generalized Fourier transform associated with the Schrödinger operator with inverse-square potentials on the half-line is the well-known Hankel transform. We present a necessary and sufficient condition for a subset $\Omega$, such that a function whose Hankel transform is supported in a given interval can be bounded, in the $L^2$-norm, from above by its restriction to $\Omega$, with a constant independent of the position of the interval.
		
		
	\end{abstract}
	
	\maketitle
	

\section{Introduction}
\subsection{Model and Result}
In this paper, we consider the observability problem for the following Schrödinger equation:
\begin{equation}\label{Se1.1}
	i \partial_t u(t,x) = H_{\nu} u(t,x), \quad x \in \mathbb{R}^+, \, t > 0, \quad u\mid_{t=0} = u_0 \in L^{2}(\mathbb{R}^{+}; \mathbb{C}),
\end{equation}
where $\mathbb{R}^+$ denotes the half-line $[0, \infty)$ and the Schrödinger operators $H_{\nu}$ are given by the Friedrichs self-adjoint extension of the following family of differential operators:
\begin{equation}
	L_{\nu} = -\partial_{x}^{2} + \left(\nu^{2} - \frac{1}{4}\right) \frac{1}{x^2},
\end{equation}
with domain $C_0^{\infty}(\mathbb{R}^{+})$ for every $\nu \geq 0$. The family of differential operators $-\partial_{x}^{2} + \left(\nu^{2} - \frac{1}{4}\right) \frac{1}{x^2}$ is quite special, appearing in numerous applications, e.g., as the radial part of the Laplacian in any dimension. Their generalized eigenfunctions can be expressed in terms of Bessel-type functions, and they have a rich and intricate theory (see \cite{KSWW, CH, EG}). Moreover, the heat flow associated with the inverse-square potential has been studied in the context of combustion theory (see \cite{LZ} and references therein). The mathematical interest in these equations mainly arises from the fact that the potential term is homogeneous of degree $-2$, meaning it scales exactly the same as the Laplacian. This scaling property makes perturbation methods ineffective in analyzing the effect of this potential. In fact, the decay behavior of such potentials represents a borderline case for the existence of global-in-time estimates for the Schrödinger equation with a potential (see \cite{RS1}).
\par The well-posedness of our model can be derived as follows. The classical Hardy inequality (see \cite{BM}) states that for every \( u \in C_0^{\infty}(\mathbb{R}^{+}) \),
\begin{equation}\label{Se1.2}
	\frac{1}{4} \int_0^{\infty} \frac{\arrowvert u(x)\arrowvert^2}{x^2} \, dx \leq \int_0^{\infty} \arrowvert u'(x)\arrowvert^2 \, dx,
\end{equation}
where \( \frac{1}{4} \) is the best possible constant.

\par This inequality implies that the operator \( -\partial_x^2 + (\nu^2 - \frac{1}{4}) \frac{1}{x^2} \) is form-bounded from below on \( C_0^{\infty}(\mathbb{R}^{+}) \). Consequently, a Friedrichs extension of this operator exists. Throughout this paper, we will only consider the Friedrichs extension \( H_\nu \) of the family \( -\partial_x^2 + (\nu^2 - \frac{1}{4}) \frac{1}{x^2} \) on \( C_0^{\infty}(\mathbb{R}^{+}) \). This extension is equivalent to the self-adjoint extension of \( -\partial_x^2 + (\nu^2 - \frac{1}{4}) \frac{1}{x^2} \) in \( L^2(\mathbb{R}^{+}) \), with the Dirichlet boundary condition at \( x = 0 \).

\par We refer to \cite{RS} for the theory of Friedrichs extensions. For a more comprehensive discussion on the theory of self-adjoint extensions for the formal operator \( -\partial_{x}^{2} + \left( \nu^2 - \frac{1}{4} \right) \frac{1}{x^2} \), we direct the reader to \cite{LJV, JS}.
\par Now, by Stone's theorem, the equation \eqref{Se1.1} has a unique weak solution in the sense that \( u(t,x) \in C(\mathbb{R}_t; L^2(\mathbb{R}^{+})) \), given by the semigroup \( e^{-itH_\nu} \) acting on \( u_0 \). That is, the solution is given by \( u(t,x) = e^{-itH_\nu} u_0 \). The main goal of this paper is to characterize the geometric conditions on a subset \(\Omega\) such that the following, the so-called exact observability inequality,
\begin{equation}\label{Se1.3}
	\int_{\mathbb{R}^{+}} \arrowvert u(0,x)\arrowvert^2 \, dx \leq C \int_0^T \int_\Omega \arrowvert u(t,x)\arrowvert^2 \, dx \, dt
\end{equation}
holds with a universal constant \( C \) depending only on \( \Omega \) and the time \( T \), whenever \( u(t,x) \) solves equation \eqref{Se1.1}. When a measurable set \( \Omega \subset \mathbb{R}^{+} \) satisfies \eqref{Se1.3} for some time \( T \), we call \( \Omega \) an \emph{observable set} at some time for equation \eqref{Se1.1}. We note that by the Hilbert uniqueness method, the exact observability of the evolution equation \eqref{Se1.1} is equivalent to the exact controllability of the adjoint system. For the concept of exact controllability, we refer to \cite{Coron}, for example.

\par Motivated by recent research \cite{HWW, JK} on the exact observability for one-dimensional free and harmonic Schrödinger equations, they independently derived the sharp geometric conditions for the observable set of one-dimensional free and harmonic Schrödinger equations. Precisely, the equations they concerned with are the following equations on \( \mathbb{R} \):
\begin{equation}\label{Se1.4}
	i \partial_t u(t,x) = - \partial_x^2 u(t,x), \quad (t,x) \in \mathbb{R} \times \mathbb{R}, \quad u(0,x) \in L^2(\mathbb{R}; \mathbb{C}),
\end{equation}
and
\begin{equation}\label{Se1.5}
	i \partial_t u(t,x) = \left( - \partial_x^2 + x^2 \right) u(t,x), \quad (t,x) \in \mathbb{R} \times \mathbb{R}, \quad u(0,x) \in L^2(\mathbb{R}; \mathbb{C}).
\end{equation}
They showed that a subset \( \Omega \subset \mathbb{R} \) is an observable set at some time for equation \eqref{Se1.4} if and only if \( \Omega \) is thick in \( \mathbb{R} \), meaning that there exist \( \gamma > 0 \) and \( L > 0 \) such that
\begin{equation}\label{Thick}
	\arrowvert \Omega \cap [x,x+L] \arrowvert \geq \gamma L, \quad \text{for each} \, x \in \mathbb{R}.
\end{equation}
Moreover, a subset \( \Omega \) is an observable set at some time for equation \eqref{Se1.5} if and only if \( \Omega \) is weakly thick in \( \mathbb{R} \), meaning that
\begin{equation*}
	\liminf_{x \to \infty} \frac{\arrowvert \Omega \cap [-x,x] \arrowvert}{x} > 0.
\end{equation*}
The authors in \cite{HWW} also prove that the thickness condition for a set is strictly stronger than the weakly thickness condition.

\par Comparing the above results, it becomes apparent that the presence of growth potential terms in the Schrödinger equation leads to a significant difference in the observable set compared to the free Schrödinger equation. The objective of this paper is to gain a better understanding of how decaying potentials, particularly the inverse-square potentials, impact the sharp geometric conditions of the observable set for the Schrödinger equation. In order to achieve this goal, we introduce the following definition.

\begin{definition}\label{Se1.6}
	A measurable subset \( \Omega \subset \mathbb{R}^+ \) is called \((r, L)\)-thick if there exist \( r > 0 \) and \( L > 0 \) such that for all \( x \geq 0 \),
	\begin{equation}\label{Se1.61}
		\arrowvert\Omega \cap [x, x + L] \arrowvert \geq r L,
	\end{equation}
	where \( \arrowvert \Omega \cap [x, x + L] \arrowvert \) denotes the Lebesgue measure of \( \Omega \cap [x, x + L] \).
\end{definition}
This definition is simply the condition \eqref{Thick} restricted to sets \( \Omega \subset \mathbb{R}^+ \). The first main result of this paper is the following theorem.

\begin{theorem}\label{T1.1}
	Let a subset \( \Omega \subset \mathbb{R}^+ \) be a measurable set. Then the following statements are equivalent:
	\begin{itemize}
		\item[(i)] The set \( \Omega \) is thick.
		\item[(ii)] The set \( \Omega \) is an observable set at some time for equation \eqref{Se1.1}.
	\end{itemize}
\end{theorem}
\begin{remark}
	For \( \nu = \frac{1}{2} \), our model \eqref{Se1.1} becomes
	\begin{equation}\label{R0}
		i\partial_t u(t,x) = -\partial_{x}^{2} u(t,x), \quad x \in \mathbb{R}^{+}, \, t > 0, \quad u \mid_{t=0} = u_0 \in L^{2}(\mathbb{R}^{+}; \mathbb{C}),
	\end{equation}
	with the function \( u(t,x) \) satisfying the Dirichlet boundary condition at \( x = 0 \). This can be viewed as the free Schrödinger equation \eqref{Se1.4} restricted to a special class of initial value functions, specifically all odd functions, in the following sense.
	
	It is well known that the explicit solution to the free Schrödinger equation \eqref{Se1.4} in terms of the Fourier transform is given by
	\begin{equation}\label{R1}
		u(x,t) = (2t)^{-\frac{1}{2}} e^{-\frac{i\pi}{4}} e^{ix^{2}/4t} \widehat{e^{i \arrowvert \cdot \arrowvert^{2}/4t} u_0}(x/2t), \quad \text{for all } x \in \mathbb{R}, t > 0.
	\end{equation}
	The explicit solution for our model \eqref{Se1.1} has a similar form, but it involves the well-known Hankel transform (see \cite{KT}), given by
	\begin{equation}\label{R2}
		u(x,t) = (2t)^{-\frac{1}{2}} e^{-\frac{i(\nu+1)\pi}{2}} e^{ix^{2}/4t} F_\nu\left( e^{i \arrowvert \cdot \arrowvert^{2}/4t} u_0 \right)\left( x/2t) \right), \quad \text{for all } x \in \mathbb{R}, t > 0,
	\end{equation}
	where the Hankel transform \( F_\nu \) is defined as
	\begin{equation*}
		F_\nu(f)(x) := \int_0^{\infty} \sqrt{xy} J_\nu(xy) f(y) \, dy, \quad f \in L^{1}(\mathbb{R}^{+}), \ \nu \geq 0,
	\end{equation*}
	with \( J_\nu \) being the Bessel function of the first kind. Now, taking \( \nu = \frac{1}{2} \), we see that \eqref{R2} is the solution of equation \eqref{R0}. Since \( F_{\frac{1}{2}} \) is simply the Fourier sine transform, defined for any odd function by
	\begin{equation*}
		F_{\frac{1}{2}}(f)(\xi) =\sqrt{\frac{2}{\pi}}\int_0^{\infty} \sin(x\xi) f(x) \, dx,
	\end{equation*}
	we conclude that \( F_{\frac{1}{2}} \) is the Fourier transform \( \mathscr{F} \) restricted to odd functions. Specifically, if \( f \) is odd and \( g = f \mid_{\mathbb{R}^{+}} \), then \( \mathscr{F}[f](\xi) = F_{\frac{1}{2}}[g](\xi) \) for \( \xi \geq 0 \). Thus, from the expressions \eqref{R1} and \eqref{R2}, our result in Theorem \ref{T1.1} for $\nu=\frac{1}{2}$ follows directly as a corollary of the aforementioned result for the free Schrödinger equation \eqref{Se1.4}. The case \( \nu \neq \frac{1}{2} \) can be seen as a perturbation of \eqref{R0} by an inverse-square potential, a decaying potential with a strong singularity at the origin. This theorem asserts that the geometric conditions under which the exact observability inequality \eqref{Se1.3} holds are invariant under these perturbations.
\end{remark}
\subsection{Main Tool: A Logvinenko-Sereda Type Theorem}
\par We point out that the characterization of observable sets for the free Schrödinger equation is derived from the famous Logvinenko-Sereda theorem in harmonic analysis, which is a type of uncertainty principle. It states that both a function and its Fourier transform cannot be arbitrarily localized, measured in terms of the smallness of their supports. Precisely, Logvinenko and Sereda first established the following theorem:

\medskip
\noindent {\bf The Logvinenko-Sereda Theorem.} Let \( J \) be an interval with \( \arrowvert J\arrowvert = b \). If \( f \in L^{p}(\mathbb{R}) \), \( p \in [1, \infty] \), and \( \text{supp}\hat{f} \subset J \), and if \( \Omega \subset \mathbb{R} \) is a thick set, then there exists a constant \( C = C(\Omega, b) \) such that
\begin{equation}\label{LS1.1}
	\| f \|_{L^{p}(\mathbb{R})} \leq C(\Omega, b) \| f \|_{L^{p}(\Omega)}.
\end{equation}
Here, \( \hat{f} \) denotes the Fourier transform of \( f \). The thickness condition on \( \Omega \) is also necessary for an inequality of the form \eqref{LS1.1} to hold, see \cite[p.113]{HJ}.

\par This result was later extended by O. Kovrijkine \cite{K}, where the following more general version was proven:

\medskip
\noindent {\bf Lemma 1.3} (Theorem 2 from \cite{K}). Let \( \{ J_k \}_{k=1}^n \) be a finite collection of intervals in \( \mathbb{R} \) with \( \arrowvert J_k\arrowvert = b \). Let \( \Omega \subset \mathbb{R} \) be a thick set. Then, there exists a constant \( C > 0 \) such that
\begin{equation}\label{Kovrijkine}
	\| f \|_{L^p(\mathbb{R})} \leq C \| f \|_{L^p(\Omega)},
\end{equation}
for all \( f \in L^p(\mathbb{R}) \), \( p \in [1, \infty] \), with \( \text{supp}\hat{f} \subset \bigcup_{k=1}^n J_k \). Moreover, \( C \) depends only on the number and size of the intervals, not on their positions.

\par In this paper, we prove a similar assertion for the Hankel transform \( F_\nu \). Using this, we deduce the observability result from Theorem \ref{T1.1} by applying the Logvinenko-Sereda type theorem for the Hankel transform \( F_\nu \). The main focus of this paper is to prove the following theorem, which constitutes the second main result.

\begin{theorem}\label{T1.2}
	Let \( \nu \geq 0 \), \( b > a \geq 0 \), and let \( f \in L^2(\mathbb{R}^+) \) with \( \text{supp}F_\nu(f) \subset [a, b] \). If \( \Omega \subset \mathbb{R}^+ \) is a thick subset, then there exists a constant \( C \) such that
	\begin{equation}\label{T1.21}
		\int_0^\infty \arrowvert f(x) \arrowvert^2 \, dx \leq C \int_\Omega \arrowvert f(x) \arrowvert^2 \, dx,
	\end{equation}
	where the constant \( C \) depends only on \( \nu \), \( \Omega \), and the length of the interval \( [a, b] \).
	
	\par Conversely, if \eqref{T1.21} holds for some constant \( C \), then \( \Omega \) must be thick.
\end{theorem}
Similar to the classical Logvinenko-Sereda theorem, A crucial property of this theorem is that the constant does not depend on the position of the interval where the support of \( F_\nu(f) \) is located, but only on the length of that interval. This property plays a crucial role in the proof of our Theorem \ref{T1.1}. We should also note that in \cite{SP}, the authors proved an inequality similar to inequality \eqref{T1.21} for the modified Hankel transform in a weighted \( L^2 \) space, with the "frequency" restricted to the interval \( [0,b] \). By exploiting the relationship between the Hankel transform \( F_\nu \) and the modified Hankel transform \( M_\nu \) (see \eqref{MH} below), we will derive the assertion for the case \( a = 0 \) in our Theorem \ref{T1.2} from their result. This is achieved after establishing an equivalent relationship between the \( \mu_\nu \) thickness condition (as used in \cite{SP}) for a set \( \Omega \) and the corresponding thickness condition in our setting. For more results on Logvinenko-Sereda-type theorems, we refer to \cite{SP,MA,GJM} and the references therein.
\subsection{Known results and contributions}
First, Theorem \ref{T1.1} addresses the observability of the Schrödinger equation with inverse square potentials on a non-compact domain. It establishes internal exact observability, which appears to be a novel result for the Schrödinger equation with these critical singular potentials. Previous work has primarily focused on exact boundary observability of the Schrödinger and wave equations with inverse square potentials on bounded domains. Key contributions in this area include the papers \cite{JE,C1,AAB}. In \cite{JE}, the authors consider the case where the singular point lies within the domain. Their results demonstrate that when the observable set takes a special form as a non-empty part of the boundary, the Schrödinger equation exhibits exact boundary observability. Meanwhile, in \cite{C1}, the author focuses on situations where the singular point is located on the boundary of the domain. The author established the exact boundary observability of the Schrödinger equation with an observable set that has the same geometrical setup as described in \cite{JE}. In \cite{AAB}, the authors consider the scenario where the inverse square singularity lies along the entire boundary of the unit ball in \( \mathbb{R}^n \) with \( n \geq 3 \). They achieve exact boundary observability with the whole boundary of the unit ball serving as the observable set. We also refer the reader to \cite{C,UE,ES,VZ}  for results related to the heat equation with inverse square potentials.
\par Secondly, considerable research has been conducted on the observability of the Schrödinger equation on compact Riemannian manifolds. In \cite{L}, it was shown that any open set satisfying the Geometric Control Condition (GCC) is an observable set at any given time. Furthermore, for manifolds with periodic geodesic flows or Zoll manifolds, it was proven in \cite{M} that the GCC is also a necessary condition. On the flat torus \( \mathbb{T}^n \), it was discovered that every non-empty open set \( E \subset \mathbb{T}^n \) is an observable set at any time, as described in \cite{HS, Ja, T} for the free Schrödinger equation, and in \cite{A2,BBZ,BZ1} for the Schrödinger equations with potentials. Moreover, \cite{BZ2} verified that for \( \mathbb{T}^2 \), each measurable set \( E \) with positive measure is an observable set at any time. \cite[Theorem 1.2]{A1} gives an observability inequality for Schrödinger equations (with some potentials) on the disk of $\mathbb{R}^{2}$. The observation is made over $\omega\times [0,T]$, where $\omega$ is an open(nonempty) subset which may not satisfy the Geometric Control Condition. For information on the observability of Schrödinger equations on negatively curved manifolds, we recommend referring to \cite{A3,DJ,DJN,J}. For recent research on the so-called observability inequality at two time points, we refer the reader to \cite{WWZ, HS}, and the references therein.
\par Thirdly, for the Schrödinger equation on non-compact Riemannian manifolds, there are relatively few existing results, as the presence of infinity in space introduces new difficulties. However, there has been growing interest in the question of observability for the Schrödinger equation in Euclidean space. In addition to the articles \cite{HWW,JK} mentioned earlier, recent work by Täufer \cite{MT} addresses the observability of the free Schrödinger equation in \( \mathbb{R}^n \), demonstrating that it is observable from any non-empty periodic open set for any positive time. This result relies on the Floquet-Bloch transform and the theory of lacunary Fourier series. This observation was later extended by Le Balc’h and Martin \cite{KLJ} to the case of periodic measurable observation sets with a periodic \( L^{\infty} \) potential in two dimensions. In a recent study, Antoine Prouf \cite{AP} investigates the Schrödinger equation in \( \mathbb{R}^n \), with \( n \geq 1 \), featuring a confining potential that grows at most quadratically. Specifically, the confining potential exhibits subquadratic growth of the form \( \lvert x \rvert^{2m} \), where \( 0 < m \leq 1 \). Through the application of semiclassical analysis techniques, Prouf establishes a refined upper bound for the optimal observation time, based on the geometry of the observable set. After completing our work, we discovered that a paper by Sun et al. \cite{SSY} had also appeared on arXiv and was recently accepted for publication in the Journal of Functional Analysis. They show that the "thickness" condition is sufficient for the observability inequality to hold for the Schrödinger equation with bounded continuous potentials. Their approach first establishes the spectral inequality for the low-frequency part using \( L^2 \)-propagation of smallness estimates, which relies on the continuity of the potential. They then deduce the spectral inequality for the high-frequency part via a perturbation method, exploiting the boundedness of the potential. They also note that their result can be extended to \( L^{\infty} \) potentials through a pointwise approximation argument. We should point out, however, that their approach does not directly apply to handle the critical singular potentials in our model. Since our approach relies on tools from harmonic analysis, we can demonstrate that the "thickness" geometric condition is also optimal in our case. We believe that, using their approach, one could similarly show that adding an \( L^{\infty} \) potential perturbation to our model would preserve the sufficiency of the thickness condition, ensuring that the observability inequality still holds. However, we do not pursue this here.
\par Finally, let us summarize the main contributions of this paper. First, we have obtained the sharp geometric condition for the observable set of the Schrödinger equation with inverse square potentials. To the best of the authors' knowledge, there has been no previous research in the literature on the internal controllability of the Schrödinger equation with inverse square potentials. Second, this paper proves a Logvinenko-Sereda type uncertainty principle related to the Hankel transform. As a byproduct of this theorem, the result extends the Logvinenko-Sereda type uncertainty principle for the modified Hankel transform, as established in \cite{SP}, to "frequency" support in arbitrary bounded intervals, rather than just intervals containing the origin.
\par {\bf Plan of the paper.} The rest of the paper is organized as follows: Section 2 provides some preliminaries, these preliminaries will appear in the proof of our theorems here and there. In Section 3, we give the proof of Theorem \ref{T1.1}. In Section 4, we prove Theorem \ref{T1.2}, which constitutes the main body of this paper. Finally, to enhance the clarity and readability of the paper, we place the proofs of Lemma \ref{L3.2} and Lemma \ref{L3.5} used in the proof of Theorem \ref{T1.2} in the appendix.
\section{Notations and Preliminaries}
\subsection{Notations}
We introduce some notations and definitions used in this paper.  We will use $C(a,b,c,...)$ to denote a constant that is allowed to depend on parameters $a$, $b$, $c$, etc. We would like to emphasize that when referring to a constant $C$ depending only on a thick set $\Omega$, it means that it depends solely on the parameters $r$ and $L$ appearing in equation \eqref{Se1.61}, that is, we will abbreviate $C(r,L)$ as $C(\Omega)$. We will denote by $\arrowvert\cdot\arrowvert$ both the Lebesgue measure on $\mathbb{R}$ and the absolute value of a real number. For real number $a$, we denote by $[a]$ as the integer part of $a$. We define $\mathbf{1}_{E}$ as the characteristic function of $E\subset \mathbb{R}^{+}$. Let \( X \) be a measure space, and let \( \mu \) be a positive measure on \( X \). For \( 1 \leq p < \infty \), we denote by \( L^p(X, \mu) \) the Banach space consisting of \( \mu \)-measurable functions on \( X \), equipped with the norms
\[
\lVert f \rVert_{L^p} = \left( \int_X \lvert f(x) \rvert^p \, d\mu(x) \right)^{1/p}.
\]
The notation \( L^p(\mathbb{R}) \) is reserved for the space \( L^p(\mathbb{R}, \lvert \cdot \rvert) \), where \( \lvert \cdot \rvert \) denotes the Lebesgue measure on $\mathbb{R}$(also denoted by \( dx \)). The same notation is reserved for \( L^p(\mathbb{R}^+) \) with the Lebesgue measure on $\mathbb{R}^+$.

For any function \( f \in L^1(\mathbb{R}, dx) \), we denote by \( \hat{f} \) the Fourier transform of \( f \), given by
\[
\hat{f}(\xi) = \mathscr{F}(f)(\xi) := \frac{1}{\sqrt{2\pi}} \int_{\mathbb{R}} e^{-ix\xi} f(x) \, dx.
\]

The Plancherel theorem for the Fourier transform states that \( \mathscr{F} \) extends to an isometric isomorphism of \( L^2(\mathbb{R}) \) onto itself. Finally, we also denote by \( f^{(n)} \) the \( n \)-th derivative of \( f \).

\subsection{Bessel equation and  Hankel transform}
In this subsection, we will examine basic facts about the Bessel equation, properties of Bessel functions, and the Hankel transform(s). There is a wealth of literature available on this topic, and we recommend readers refer to \cite{N} for further details and references. The Bessel equation is a prominent ordinary differential equation that applies to every $m\in \mathbb{R}$,
\begin{equation*}
	x^2y''(x)+xy'(x)+(x^2-m^2)y=0.
\end{equation*}
It is well known that the Bessel function of the first kind $J_m$ and $J_{-m}$ solve this equation.
\begin{equation}\label{Bessel}
	J_m(x)=\sum_{n=0}^{\infty}\frac{(-1)^{n}(x/2)^{2n+m}}{n!\Gamma(n+m+1)}.
\end{equation}
We have the following identity for $x>0$, $m>-\frac{1}{2}$, see \cite[p.580]{GL}:
\begin{equation}\label{Sec2.1}
	J_m(x)=\sqrt{\frac{2}{\pi x}}cos(x-\frac{\pi m}{2}-\frac{\pi}{4})+R_{m}(x),
\end{equation}
where $R_{m}$ is given by 
\begin{equation}
	\begin{aligned}
		R_{m}(x)&=\frac{(2\pi)^{-\frac{1}{2}}x^{m}}{\varGamma(m+\frac{1}{2})}e^{i(x-\frac{\pi m}{2}-\frac{\pi}{4})}\int^{\infty}_{0}e^{-xt}t^{m+\frac{1}{2}}[(1+\frac{it}{2})^{m-\frac{1}{2}}-1]\frac{dt}{t}\\
		&+\frac{(2\pi)^{-\frac{1}{2}}x^{m}}{\varGamma(m+\frac{1}{2})}e^{-i(x-\frac{\pi m}{2}-\frac{\pi}{4})}\int^{\infty}_{0}e^{-xt}t^{m+\frac{1}{2}}[(1-\frac{it}{2})^{m-\frac{1}{2}}-1]\frac{dt}{t},
	\end{aligned}
\end{equation}
and satisfies 
\begin{equation}\label{Sec2.2}
	\arrowvert R_{m}(x)\arrowvert\leq C_{m}x^{-\frac{3}{2}},
\end{equation} 
whenever $x\geq1$.
\par We point out that identity \eqref{Sec2.1} and the estimate \eqref{Sec2.2} are the starting points for proving Theorem \ref{T1.2} later.
\par Now, we come to define the Hankel transform $F_\nu(f)$ of order $\nu(\nu>-\frac{1}{2})$ of any function $f\in L^{1}(\mathbb{R}^{+})$ by
\begin{equation}\label{Sec2.1.6}
	F_\nu(f)(x):=\int_{0}^{\infty}\sqrt{xy}J_\nu(xy)f(y)dy,\  x\in \mathbb{R}^{+},
\end{equation}
where $J_\nu$ is the Bessel function of the first kind given above.
The Plancherel theorem for the Hankel transform states that $F_\nu$ extends to an isometric isomorphism of $L^{2}(\mathbb{R}^{+})$ onto itself. And the inverse Hankel transform has the symmetric form (see, for example,\cite{N})
\begin{equation}
	f(x):=\int_{0}^{\infty}\sqrt{xy}J_\nu(xy)F_\nu(f)(y)dy, \  x\in \mathbb{R}^{+}.
\end{equation}
\subsection{Two technical lemmas}
The first lemma is a crucial technique that will be employed later to prove the equivalence between the thickness condition and the \( \mu_{\nu} \)-thickness condition (for the precise definition of \( \mu_{\nu} \)-thickness, see \eqref{DM} below).
\begin{lemma}\label{Technical1.1}
	For any fixed $x\in [0,+\infty)$ and a constant $L>0$, assume $A\subset[x,x+L]$ is a Lebesgue measurable subset of $[x,x+L]$, and $\arrowvert A\arrowvert\geq\gamma L$, $0<\gamma\leq1$, then for every increasing  function $f$, we have
	\begin{equation}
		\int_{A}f(t)dt\geq\int_{[x,x+\gamma L]}f(t)dt.
	\end{equation}
\end{lemma}
\begin{proof}
	By using the fact that the function $f$ is increasing on $[0,+\infty]$, we observe that
	\begin{equation*}
		\begin{aligned}&\int_{A}f(t)dt\\
			&=\int_{A\cap[x,x+\gamma L]}f(t)dt+\int_{A\cap([x,x+L]\backslash[x,x+\gamma L])}f(t)dt\\
			&\geq\int_{A\cap[x,x+\gamma L]}f(t)dt+\arrowvert A\cap([x,x+L]\setminus[x,x+\gamma L])\arrowvert f(x+\gamma L)\\
			&\geq\int_{A\cap[x,x+\gamma L]}f(t)dt+\frac{\arrowvert A\cap([x,x+L]\setminus[x,x+\gamma L])\arrowvert}{\arrowvert[x,x+\gamma L]\cap([x,x+L]\setminus A)\arrowvert}\int_{[x,x+\gamma L]\cap([x,x+L]\setminus A)}f(t)dt.
		\end{aligned}
	\end{equation*}
	Then by using that $\arrowvert A\arrowvert\geq\gamma L$, we notice that
	$$\arrowvert A\cap([x,x+L]\setminus[x,x+\gamma L])\arrowvert=
	\arrowvert A\arrowvert-\arrowvert A\cap[x,x+\gamma L]\arrowvert$$
	$$\geq \arrowvert [x,x+\gamma L]\arrowvert-\arrowvert A\cap[x,x+\gamma L]\arrowvert=\arrowvert[x,x+\gamma L]\cap([x,x+L]\setminus A)\arrowvert.$$
	It follows that,
	\begin{equation*}
		\begin{aligned}
			&\int_{A}f(t)dt\\
			&\geq\int_{A\cap[x,x+\gamma L]}f(t)dt+\frac{\arrowvert[x,x+\gamma L]\cap([x,x+L]\setminus A)\arrowvert}{\arrowvert[x,x+\gamma L]\cap([x,x+L]\setminus A)\arrowvert}\int_{[x,x+\gamma L]\cap([x,x+L]\setminus A)}f(t)dt\\
			&=\int_{A\cap[x,x+\gamma L]}f(t)dt+\int_{[x,x+\gamma L]\cap([x,x+L]\setminus A)}f(t)dt=\int_{[x,x+\gamma L]}f(t)dt.
		\end{aligned}
	\end{equation*}
	This ends the proof of the lemma.
\end{proof}
The second lemma is a key technique that we will use later in proving Theorem \ref{T1.2}. Let us assume $\Omega$ is a thick set, meaning that there exist constants $L$ and $r$ such that for any $x \in \mathbb{R}^{+}$,
$$\arrowvert \Omega \cap [x, x+L]\arrowvert \geq rL.$$
We now propose the claim that for any constant $c \in \mathbb{R}^{+}$, the set $\Omega^{\prime} = \Omega \setminus [0, c]$ is also a thick set.
\begin{lemma}\label{Technical1.2}
	For a given \( (r,L) \)-thick set \( \Omega \), define \( \Omega' := \Omega \setminus [0,c] \), where \( c \) is a given nonnegative constant. Then, there exist constants \( L_1 \) and \( r_1 \) such that for all \( x \in \mathbb{R}^+ \),
	\[
	\left\arrowvert \Omega' \cap [x, x + L_1] \right\arrowvert \geq r_1 L_1.
	\]
	Furthermore, the constants \( L_1 \) and \( r_1 \) can be explicitly given as
	\[
	L_1 = \left( \frac{c}{L} + 2 \right) L, \quad r_1 = \frac{r}{\frac{c}{L} + 2},
	\]
	which depend only on \( L \), \( r \), and \( c \).
\end{lemma}
\begin{proof} 
	Let \( L_1 = \left( \frac{c}{L} + 2 \right) L \) and \( r_1 = \frac{r}{\frac{c}{L} + 2} \). We now proceed to prove that
	\begin{equation}
		\left\arrowvert \Omega' \cap [x, x + L_1] \right\arrowvert\geq r_1 L_1, \quad \forall x \in \mathbb{R}^+.
	\end{equation}
	For this purpose, we consider two cases.\\
	
	\textbf{Case 1:} Assume \( x < c \). We further divide this case into two subcases.\\
	
	\textbf{Subcase 1.1:} For \( c \geq L \), since \( L_1 = \left( \frac{c}{L} + 2 \right) L > c + L \), for \( 0 \leq x < c \), we have
	\[
	\left\arrowvert \Omega' \cap [x, x + L_1] \right\arrowvert = \left\arrowvert \Omega \cap [c, x + L_1] \right\arrowvert \geq \left\arrowvert \Omega \cap [c, c + L] \right\arrowvert \geq rL = \frac{r}{\frac{c}{L} + 2} \left( \frac{c}{L} + 2 \right) L.
	\]
	
	\textbf{Subcase 1.2:} For \( c < L \), we observe that
	\[
	\left\arrowvert \Omega' \cap [0, 2L] \right\arrowvert = \left\arrowvert\Omega \cap [c, 2L] \right\arrowvert\geq \left\arrowvert \Omega \cap [L, 2L] \right\arrowvert \geq rL.
	\]
	With our choice of \( L_1 \), we have \( L_1 \geq 2L \). Since \( x < c \), it follows that
	\[
	\left\arrowvert \Omega' \cap [x, L_1] \right\arrowvert \geq \left\arrowvert \Omega' \cap [c, L_1] \right\arrowvert \geq \left\arrowvert \Omega' \cap [c, 2L] \right\arrowvert \geq rL = \frac{r}{\frac{c}{L} + 2} \left( \frac{c}{L} + 2 \right) L.
	\]
	
	\textbf{Case 2:} Assume \( x \geq c \). This case is easier, as we observe that
	\[
	\left\arrowvert \Omega' \cap [x, x + L_1] \right\arrowvert = \left\arrowvert \Omega \cap [x, x + L_1] \right\arrowvert \geq \left\arrowvert \Omega \cap [x, x + L] \right\arrowvert \geq \frac{r}{\frac{c}{L} + 2} \left( \frac{c}{L} + 2 \right) L.
	\]
	
	Hence, \( \Omega' \) is a thick set with constants \( L_1 = \left( \frac{c}{L} + 2 \right) L \) and \( r_1 = \frac{r}{\frac{c}{L} + 2} \), depending only on \( L \), \( r \), and \( c \).
\end{proof}

A key step in the proof of Theorem \ref{T1.2} is to choose 
$c$ in above lemma as a constant that depends only on the length of the interval $J$, where the support of $F_\nu(f)$ is located.
\section{Proof of Theorem \ref{T1.1}}
\par To prove Theorem \ref{T1.1}, we first temporarily assume that Theorem \ref{T1.2} holds. In this section, we will see how Theorem \ref{T1.1} can be derived from Theorem \ref{T1.2}. We adopt the strategy from \cite{HWW}, where we begin by introducing a resolvent condition on the observability for a certain evolution equation. This approach is another version of Theorem 5.1 in \cite{ML}, and readers can also refer to \cite{HWW} for further details and references.
\begin{proposition}\label{P1}
	Let $E\subset\mathbb{R}^{+}$ be a measurable set. Then the following statements are equivalent:\\
	(i) The set $E$ is an observable set at some time for equation \eqref{Se1.1}.\\
	(ii) There is $M>0$ and $m>0$ such that
	\begin{equation}\label{P1.1}
		\| u\|^{2}_{L^{2}(\mathbb{R}^{+})}\leq M\| (H_{\nu}-\lambda)u\|^{2}_{L^{2}(\mathbb{R}^{+})}+m\| u\|^{2}_{L^{2}(E)},
	\end{equation}
	for all $u\in D(H_{\nu})$ and $\lambda\in \mathbb{R}$.
\end{proposition}
We also need the following spectral estimates which given by Miller \cite[Corollary 2.17]{ML1}.
\begin{proposition}\label{P2}
	If equation \eqref{Se1.1} is exactly observable from a measurable subset $\Omega\subset \mathbb{R}^{+}$ at some time $T>0$, then there exist some positive constants $k>0$ and $D>0$ such that 
	\begin{equation}\label{P2.1}
		\forall \lambda\in \mathbb{R},\,\, \forall f\in {\mathbf{1}}_{\{\arrowvert H_{\nu}-\lambda\arrowvert\leq\sqrt{D}\}}(L^2(\mathbb{R}^{+})),\,\,\,\,\,\, \Arrowvert f\Arrowvert_{L^2(\mathbb{R}^{+})}\leq \sqrt{k}\Arrowvert f\Arrowvert_{L^2(\Omega)}.
	\end{equation}
	\\Conversely, when the spectral estimate \eqref{P2.1} holds for some $k>0$ and $D>0$, then the system \eqref{Se1.1} is exactly observable from $\Omega$ at any time 
	\begin{equation}
		T>\pi \sqrt{\frac{1+k}{D}}.
	\end{equation}
\end{proposition}
The operators ${\mathbf{1}}_{\{\arrowvert H_{\nu}-\lambda\arrowvert\leq\sqrt{D}\}}$ appearing in \eqref{P2.1} are defined by functional calculus.
Finally we refer to Theorem 5.2 in reference \cite{LJV}, which asserts that the Hankel transform $F_\nu$ diagonalizes $H_\nu$.
\begin{proposition}\label{P3}
	$F_\nu$ is a unitary involution on $L^{2}(0,\infty)$ diagonalizing  $H_\nu$, precisely,
	\begin{equation}
		F_\nu H_\nu F_\nu^{-1}=Q^{2},
	\end{equation}
	where $Q^{2}f(x):=x^{2}f(x)$.
\end{proposition}
To get \eqref{P1.1} in Proposition \ref{P1}, we have the following proposition.
\begin{proposition}\label{P4}
	Let $\Omega\subset\mathbb{R}^{+}$ be thick, $\nu\geq 0$. There exists $C>0$ depending on $\Omega$, $\nu$, such that for all $f\in L^{2}(\mathbb{R}^{+})$, $\lambda\geq 0$,
	\begin{equation}\label{P4.1}
		C\| f\|^{2}_{L^{2}(\mathbb{R}^{+})}\leq (1+\lambda)^{-1}\| [(H_{\nu}+1)-\lambda] f\|^{2}_{L^{2}(\mathbb{R}^{+})}+\| f\|^{2}_{L^{2}(\Omega)}.
	\end{equation}
\end{proposition}
\begin{proof}
	The proof presented here is inspired by Green's argument for operator $-\partial_x^{2}$ \cite{G}.
	Let $g\in L^{2}(\mathbb{R}^{+})$, such that
	$$supp\,F_{\nu}(g)(\xi)\subset A_{\lambda}:=\{\xi\in \mathbb{R}^{+}:\arrowvert(\xi^{2}+1)^{\frac{1}{2}}-\lambda^{\frac{1}{2}}\arrowvert\leq 1\}. $$We know that $A_{\lambda}$ is an interval in $\mathbb{R}^{+}$ of length no more than 4. Therefore, for a thick set $\Omega\subset\mathbb{R}^{+}$, according to Theorem \ref{T1.2}, there exists $C>0$ (independent of $\lambda$ and $g$) such that
	\begin{equation*}
		\| g\|_{L^{2}(\mathbb{R}^{+})}\leq C\| g\|_{L^{2}(\Omega)}.
	\end{equation*}
	Denote by $P_{\lambda}$ the projection $P_{\lambda}f=F_{\nu}(\mathbf{1}_{A_{\lambda}}F_{\nu}(f))$, then for $f\in L^{2}(\mathbb{R}^{+})$,
	\begin{equation*}	
		\begin{aligned}
			\| f\|^{2}_{L^{2}(\mathbb{R}^{+})}&=\| P_{\lambda}f\|^{2}_{L^{2}(\mathbb{R}^{+})}+\|(I-P_{\lambda}) f\|^{2}_{L^{2}(\mathbb{R}^{+})}\\
			&\leq C\| P_{\lambda}f\|^{2}_{L^{2}(\Omega)}+\|(I-P_{\lambda}) f\|^{2}_{L^{2}(\mathbb{R}^{+})}\\
			&=C\|f-(I-P_{\lambda}) f\|^{2}_{L^{2}(\Omega)}+\|(I-P_{\lambda}) f\|^{2}_{L^{2}(\mathbb{R}^{+})}\\
			&\leq 2C\| f\|^{2}_{L^{2}(\Omega)}+2C\|(I-P_{\lambda}) f\|^{2}_{L^{2}(\Omega)}+\|(I-P_{\lambda}) f\|^{2}_{L^{2}(\mathbb{R}^{+})}\\
			&\leq2C\| f\|^{2}_{L^{2}(\Omega)}+(2C+1)\|(I-P_{\lambda}) f\|^{2}_{L^{2}(\mathbb{R}^{+})}.
		\end{aligned}
	\end{equation*}
	It remains to estimate the final term. Since $F_\nu$ is a unitary transform and the fact	
	$$[(H_{\nu}+1)-\lambda]f(x)=F_\nu([(\xi^{2}+1)-\lambda]F_{\nu}(f)(\xi))(x),$$ 
	then applying Lemma 1 in \cite{G},\\
	\begin{equation*}	
		\begin{aligned}
			\|[(H_{\nu}+1)-\lambda]f\|^{2}_{L^{2}(\mathbb{R}^{+})}&=
			\int_{\mathbb{R}^{+}}[(\xi^{2}+1)-\lambda]^{2}\arrowvert F_{\nu}(f)\arrowvert^{2}d\xi\\
			&\geq \int_{A_{\lambda}^{c}}[(\xi^{2}+1)-\lambda]^{2}\arrowvert F_{\nu}(f)\arrowvert^{2}d\xi\\
			&\geq C(\lambda+1)\int_{A_{\lambda}^{c}}\arrowvert F_{\nu}(f)\arrowvert^{2}d\xi\\
			&=C(\lambda+1)\|(I-P_{\lambda}) f\|^{2},
		\end{aligned}
	\end{equation*}
	which indicates that there exists $C>0$ such that
	$$C\| f\|_{L^{2}(\mathbb{R}^{+})}\leq (1+\lambda)^{-1}\| [(H_{\nu}+1)-\lambda]f\|_{L^{2}(\mathbb{R}^{+})}+\| f\|_{L^{2}(\Omega)}.$$
\end{proof}
Based on the previous propositions, we will now prove our main Theorem \ref{T1.1}.\\
\par \textbf{Proof of Theorem \ref{T1.1}}\\
{\bf Step 1.} We show that (i) of Theorem \ref{T1.1} implies (ii) of Theorem \ref{T1.1}.
\par Suppose that $\Omega$ is thick. Arbitrarily fix $\lambda\in \mathbb{R}$, and $u\in D(H_{\nu})$, in the case that $\lambda\geq0$, according to Proposition \ref{P4}, there exist constants $C(\nu,\Omega)$ and $m(\nu,\Omega)$, such that 
\begin{equation}\label{Pr1.2.1}
	\| f\|^{2}_{L^{2}(\mathbb{R}^{+})}\leq C(\nu,\Omega)(\| (H_{\nu}-\lambda)u\|^{2}_{L^{2}(\mathbb{R}^{+})}+m(\nu,\Omega)\| u\|^{2}_{L^{2}(\Omega)}),
\end{equation}
for $\lambda<0$, we have 
$$\| H_{\nu}u\|^{2}_{L^{2}(\mathbb{R}^{+})}=\| \xi^{2}F_\nu u\|^{2}_{L^{2}(\mathbb{R}^{+})}\leq\| (\xi^{2}-\lambda)u\|^{2}_{L^{2}(\mathbb{R}^{+})}=\| (H_{\nu}-\lambda)u\|^{2}_{L^{2}(\mathbb{R}^{+})},$$
which, along with (\ref{Pr1.2.1}) where $\lambda=0$, shows that \eqref{P4.1} also holds for all $\lambda<0$.
\par Therefore, by Proposition \ref{P1}, $\Omega$ is an observable set at some time for equation \eqref{Se1.1}.\\
\par {\bf Step 2.} We show that (ii) of Theorem \ref{T1.1} implies (i) of Theorem \ref{T1.1}.
Since the Hankel transform $F_\nu$ is a unitary operator, we know by above Proposition \ref{P3} that the operator \( H_{\nu}\) is unitarily equivalent to the multiplication operator $x^{2}$. Thus by functional calculus for self-adjoint operator, for every $f\in L^{2}(\mathbb{R}^{+})$, we obtain
\begin{equation}
	\mathbf{1}_{\{\arrowvert H_\nu-\lambda\arrowvert\leq\sqrt{D}\}}f(x)=F_{\nu}\left(\mathbf{1}_{\{\arrowvert \xi^2-\lambda\arrowvert\leq\sqrt{D}\}}F_{\nu}f(\xi)\right)(x),
\end{equation}
which implies
\begin{equation}
	\mathbf{1}_{\{\arrowvert H_\nu-\lambda\arrowvert\leq\sqrt{D}\}}(L^2(\mathbb{R}^{+})=\{f\in L^2(\mathbb{R}^{+}): supp\,F_\nu(f)\subset\{\xi\in \mathbb{R}^{+},\arrowvert\xi^2-\lambda\arrowvert\leq \sqrt{D}\}\}.
\end{equation} 
\par We now suppose that $\Omega$ satisfies (ii) of Theorem \ref{T1.1}. We deduce from Proposition \ref{P2} that there exist some positive constants $k>0$ and $D>0$ such that for all $\lambda\in \mathbb{R}$ and $f\in L^2(\mathbb{R}^{+})$,
\begin{equation}\label{Pr1.2.2}
	supp\,F_\nu(f)\subset\{\xi\in \mathbb{R}^{+},\arrowvert\xi^2-\lambda\arrowvert\leq\sqrt{D}\}\,\,\Rightarrow\,\,\Arrowvert f\Arrowvert\leq \sqrt{k}\Arrowvert f\Arrowvert_{L^2(\Omega)}.
\end{equation}
While taking $\lambda=0$, assertion \eqref{Pr1.2.2} becomes
\begin{equation*}
	supp\,F_\nu(f)\subset[0,D^{\frac{1}{4}}]\,\,\Rightarrow\,\,\Arrowvert f\Arrowvert\leq \sqrt{k}\Arrowvert f\Arrowvert_{L^2(\Omega)}.
\end{equation*}
Then by Theorem \ref{T1.2} with $h=D^{\frac{1}{4}}$, we obtain that the control subset $\Omega$ must be thick.
\\Thus, we have completed the proof of Theorem \ref{T1.1}. $\hfill{\Box}$
\begin{remark} We should point out that the above discussion does not provide an exact estimate for the control time \( T \). We would like to emphasize that if we assume that $\sigma_{ess}(H)=\emptyset$ and $m$ is a fixed constant, while the constant $M$ in the resolvent inequality \eqref{P1.1} depends on $\lambda$ and satisfies the convergence condition $M(\lambda) \rightarrow 0$ as $\arrowvert\lambda\arrowvert \rightarrow +\infty$, then $\Omega$ becomes the observable set at any time $T>0$ for the system \eqref{Se1.1} (see \cite [Corollary 2.14]{ML1}). It is also worth mentioning the works \cite{DJ} and \cite{DJN}, where the authors obtained a quantitative resolvent inequality of the Hautus type on negatively curved compact manifolds.
\end{remark}
\begin{remark}
	It is worth noting that if we consider $s\geq\frac{1}{2}$ and replace the operator $H_{\nu}$ with its fractional counterpart  $H_{\nu}^{s}$, then the conclusions of Theorem \ref{T1.1} remain valid without any modifications. The proof follows the same reasoning as above. However, for $0<s<\frac{1}{2}$, the measure $\arrowvert [(-\sqrt{D}+\lambda)^{\frac{1}{2s}},(\sqrt{D}+\lambda)^{\frac{1}{2s}}]\arrowvert$ no longer has a fixed upper bound, causing Proposition \ref{P2} to fail in these cases. Interested readers can refer to \cite{JK} for geometric conditions related to the controllability of fractional free and harmonic Schr\"{o}dinger equations on one-dimensional Euclidean space, and to \cite{M2} for observability inequalities concerning problems involving fractional Schr\"{o}dinger operators $(-\Delta)^{\alpha/2}+V,\,\alpha>0$ on a compact Riemannian manifold.
\end{remark}

\section{Proof of Theorem \ref{T1.2}}
In order to prove Theorem \ref{T1.2}, it is sufficient to prove the following two lemmas.
\begin{lemma}\label{KeyL1}
	Let $\nu\geq 0 $ and $f\in{L^{2}(\mathbb{R}^{+})}$, such that supp $F_{\nu}(f)\subset[0,h]$. Then there is a constant $C$ such that the following inequality holds for all such $f$,
	\begin{equation}\label{KeyL1.1}
		\int_{0}^{\infty}\arrowvert f(x)\arrowvert^{2}dx\leq C\int_{\Omega}\arrowvert f(x)\arrowvert^{2}dx, 
	\end{equation}
	if and only if $\Omega$ is thick, where $C$ depends only on $\Omega$, $\nu$ and $h$.
\end{lemma}
\begin{lemma}\label{KeyL2}
	Let $\nu\geq 0 $, $b>a>0$, given a thick subset $\Omega\subset \mathbb{R}^{+}$. Then there is a constant $C$ depends only on $\Omega$, $\nu$ and $b-a$ such that 
	\begin{equation}\label{KeyL1.2}
		\int_{0}^{\infty}\arrowvert f(x)\arrowvert^{2}dx\leq C\int_{\Omega}\arrowvert f(x)\arrowvert^{2}dx 
	\end{equation}
	holds for all functions $f\in{L^{2}(\mathbb{R}^{+})}$ with supp $F_{\nu}(f)\subset[a,b]$.
\end{lemma}
It is easy to see that Theorem \ref{T1.2} is a combination of Lemma \ref{KeyL1} and Lemma \ref{KeyL2}. Next, We will prove Lemma \ref{KeyL1} and Lemma \ref{KeyL2} one by one.
\par Before proceed, we should emphasize that we will use Lemma \ref{KeyL1} to derive Lemma \ref{KeyL2}, but this reasoning process is not as straightforward as in the case of the Fourier transform. For the Fourier transform, if formula \eqref{LS1.1} holds for the interval \( J = [0, h] \), then for \( J = [a, b] \) with \( b > a > 0 \), and assuming \( \text{supp} \, \hat{f} \subset [a, b] \), we have \( \text{supp} \, \hat{g}(\xi) := \hat{f}(\xi + a) \subset [0, b - a] \). Thus, \eqref{LS1.1} holds for the function \( g \). Moreover, since \( g(x) = e^{-iax} f(x) \), it follows that \( \arrowvert g(x)\arrowvert= \arrowvert f(x)\arrowvert \), meaning that formula \eqref{LS1.1} also holds for \( f \), with a constant \( C \) depending only on the length of \( J \). However, in our case, the modulus \( \arrowvert F_\nu(f)\arrowvert \) does not exhibit translation invariance, meaning that \( \arrowvert F_\nu(f)(x)\arrowvert \neq \arrowvert F_\nu(f(\bullet + a))(x)\arrowvert \). This lack of invariance makes the reasoning in Lemma \ref{KeyL2} significantly more complex.

\par Now, we proceed to prove Lemma \ref{KeyL1}. Before proceeding with the proof of Lemma \ref{KeyL1}, we introduce two modified definitions. First, the modified Hankel transform $M_\nu$ is given by (see \cite{SP})
\begin{equation}\label{MH}
	M_\nu(f)(x):=\int_{0}^{\infty}(xy)^{-\nu}J_\nu(xy)f(y)d\mu_\nu(y), \ x\in \mathbb{R}^{+},
\end{equation}
where $d\mu_\nu(y)=y^{2\nu+1}dy$. It is well known that $M_\nu$ extends to an isometric isomorphism from the weighted  $L^{2}(\mathbb{R}^{+}, x^{2\nu+1} dx)$ onto itself, with a symmetric inverse: $M_\nu^{-1}=M_\nu$. A modified definition of thickness is given as below (see \cite{SP}).
\begin{definition}\label{DM}
	Let $\nu\geq 0$, a measurable subset $\Omega\subset \mathbb{R}^{+}$ is called (r,L)-$\mu_{\nu}$ thick if there exist $r>0$ and $L>0$, such that for all $x\geq0$,
	\begin{equation*}
		\mu_{\nu}(\Omega\cap [x,x+L])\geq r\mu_{\nu}([x,x+L]).
	\end{equation*}
\end{definition}
Thanks to Theorem 4.3 and Lemma 4.2 in \cite{SP}, the proof of Lemma \ref{KeyL1} saves some trouble. By combining Theorem 4.3 and Lemma 4.2 in \cite{SP}, we get\\
\\{\bf Proposition}(Thm 4.3 and Lemma 4.2 from \cite{SP})
Let $\nu\geq 0 $ and $f\in L^{2}(\mathbb{R}^{+},x^{2\nu+1} dx)$, such that supp $M_{\nu}(f)\subset[0,h]$. Then, there exists a constant $C$ such that the following inequality holds for all such functions $f$:
\begin{equation} 
	\int_{0}^{\infty}\arrowvert f(x)\arrowvert^{2}du_{\nu}(x)\leq C\int_{\Omega}\arrowvert f(x)\arrowvert^{2}du_{\nu}(x),
\end{equation}
if and only if $\Omega$ is $u_{\nu}$-thick. The explicit form of $C$ is given by:
$$C=\frac{2}{3}\left(\frac{r}{300\times 9^{\nu}}\right)^{(160\sqrt{3}\pi/\ln{2})hL+\nu(\ln{3}/\ln{2})+1}.$$
\begin{remark}
	We observe that the constant $C$ depends solely on the parameters $h$ and $\nu$, as well as $r$ and $L$ from the definition of  $u_{\nu}$ thick. This fact is crucial for our subsequent analysis.
\end{remark}
\begin{remark}
	The authors in \cite{SP} did not address the scenario where the support of $M_{\nu}(f)$ is contained in $[a,h]$, with $a>0$. There are two main reasons why this generalization is not possible. Firstly, the proof relies on a variant of Bernstein's inequality for $M_\nu$, which can only be established when $a=0$. Secondly, the modulus $\arrowvert M_\nu(f)\arrowvert$ does not exhibit translation invariance as mentioned above.
\end{remark}
\par Next, we claim that a similar estimate holds for functions with \(\operatorname{supp} F_{\nu}(f) \subset [0,h]\).

\begin{lemma}\label{L3.1}
	Let \(\nu \geq 0\) and \(f \in L^2(\mathbb{R}^+)\) such that \(\operatorname{supp} F_{\nu}(f) \subset [0,h]\). Then there exists a constant \(C\) such that the following inequality holds for all such \(f\),
	\begin{equation}\label{Le3.1}
		\int_0^\infty \arrowvert f(x)\arrowvert^2 \, dx \leq C \int_\Omega \arrowvert f(x)\arrowvert^2 \, dx,
	\end{equation}
	if and only if \(\Omega\) is \(u_{\nu}\)-thick. The constant \(C\) is given by
	\begin{equation}\label{Constant}
		C = \frac{2}{3} \left( \frac{r}{300 \times 9^\nu} \right)^{\left( \frac{160 \sqrt{3} \pi}{\ln 2} \right) h L + \nu \left( \frac{\ln 3}{\ln 2} \right) + 1}.
	\end{equation}
\end{lemma}

\begin{proof}
	It can be easily verified using the definitions of \(F_\nu\) and \(M_\nu\) that
	\begin{equation}\label{Sec3.1}
		F_\nu(f)(x) = x^{\nu + \frac{1}{2}} M_\nu(y^{-\nu - \frac{1}{2}} f), \quad \forall y \in \mathbb{R}^+.
	\end{equation}
	We define the operator \(A_\nu(f)(r) = r^{-\nu - \frac{1}{2}} f(r)\), which has the following properties:
	\begin{equation*}
		A_\nu: L^2(\mathbb{R}^+, dr) \longrightarrow L^2(\mathbb{R}^+, r^{2\nu + 1} dr) \quad \text{is unitary,}
	\end{equation*}
	since
	\begin{equation*}
		\arrowvert A_\nu(f)\arrowvert^2 \, du_\nu(x) = \arrowvert f\arrowvert^2 \, dx.
	\end{equation*}
	Now, let us consider any \(f \in L^2(\mathbb{R}^+, dr)\) with \(\operatorname{supp} F_\nu(f) \subset [0,h]\). We assume that
	\begin{equation}\label{Le3.2}
		\int_0^\infty \arrowvert f(x)\arrowvert^2 \, dx \leq C \int_\Omega \arrowvert f(x)\arrowvert^2 \, dx,
	\end{equation}
	where \(\Omega\) is a subset of \(\mathbb{R}^+\). Then, it is equivalent to
	\begin{equation}\label{Le3.3}
		\int_0^\infty \arrowvert A_\nu f(x)\arrowvert^2 \, du_\nu(x) \leq C \int_\Omega \arrowvert A_\nu f(x)\arrowvert^2 \, du_\nu(x).
	\end{equation}
	By utilizing formula \eqref{Sec3.1}, we can derive that
	\begin{equation}\label{Le3.4}
		\operatorname{supp} F_\nu(f) \subset [0,h] \quad \Leftrightarrow \quad \operatorname{supp} M_\nu(A_\nu f) \subset [0,h].
	\end{equation}
	Hence, combining \eqref{Le3.4} with the proposition above, we conclude that \eqref{Le3.3} holds if and only if \(\Omega\) is \(u_\nu\)-thick.
	
	Finally, by the equivalence of \eqref{Le3.2} and \eqref{Le3.3}, we obtain the assertion of Lemma \ref{L3.1}.
\end{proof}

Now, the Lemma \ref{KeyL1} that we need to prove is a corollary of Lemma \ref{L3.1} and the subsequent lemma which assert that the $u_{\nu}$-thickness condition for a set $\Omega$ is equivalent to the thickness condition for it.
\begin{lemma}\label{L3.2}
	Let $\Omega\subset \mathbb{R}^{+}$ be a measurable set, $\nu\geq0$. Then the following statements are equivalent: \\
	\\
	(i)There exist constants $r>0$ and $L>0$ such that $\Omega$ is $(r,L)$-$u_{\nu}$ thick. \\
	\\
	(ii)There exist constants $r_1>0$ and $L_1>0$ such that $\Omega$ is $(r_1,L_1)$-thick.
	\\
	\\where $(r,L)$ and $(r_1,L_1)$ depend only on each other and $\nu$.
\end{lemma}
Since the proof of this lemma is basic, in order to make the paper clearer and more readable, we place its proof in the appendix.\\
\\
\textbf{Proof of Lemma \ref{KeyL1}}
Assuming there is a set $\Omega$ such that inequality \eqref{KeyL1} holds with some constant $C$, by Lemma \ref{L3.1}, then iff there exist constant $r>0$ and $L>0$ such that $\Omega$ is $(r,L)$-$u_{\nu}$ thick, and by Lemma \ref{L3.2}, iff there exist constant $r_1>0$ and $L_1>0$ such that  $\Omega$ is $(r_1,L_1)$-thick. Since $(r,L)$ and $(r_1,L_1)$ depend only on each other and $\nu$, then the explicit expression \eqref{Constant} implies that the constant $C$ depends only on $\Omega$, $\nu$ and $h$. Thus we obtain the assertion of Lemma \ref{KeyL1}. $\hfill{\Box}$
\begin{remark}
	Directly proving inequality \eqref{KeyL1.1} for functions with support \( F_{\nu}(f) \subset [0, h] \) presents certain challenges. This is because a crucial step in the proof, based on Kovrijkine's method, relies on the analytic properties of the function \( f \). However, in this case, \( f \) is not an analytic function, but rather a generalized power series near the origin. Therefore, the result established in \cite{SP} for the modified Hankel transform \( M_\nu \) plays a key role in proving inequality \eqref{KeyL1.1}.
\end{remark}
Next, we proceed to prove Lemma \ref{KeyL2}. To do so, we first need to establish several intermediate lemmas. We begin by defining a family of operators in \( L^2(\mathbb{R}^{+}) \) as follows:
\begin{equation}\label{TB}
	T_{\beta}(f)(x) = \sqrt{\frac{2}{\pi}} \int_{0}^{\infty} \cos(xy - \beta) f(y) \, dy, \quad f \in L^2(\mathbb{R}^{+}).
\end{equation}

Due to the periodicity of \( \cos(x) \), it suffices to consider \( \beta \in \left(-\frac{\pi}{2}, \frac{\pi}{2}\right] \). In particular, for \( \beta = 0 \), \( T_0 \) corresponds to the Fourier-cosine transform, and for \( \beta = \frac{\pi}{2} \), \( T_{\frac{\pi}{2}} \) is the Fourier-sine transform.

First, we need to establish a result concerning the boundedness and invertibility of these operators.
\begin{lemma}\label{L3.5}
	For every $\beta\in \mathbb{R}$, the operators $T_\beta$ defined by \eqref{TB} extends to bounded invertible maps on $L^2(\mathbb{R}^{+})$. Moreover the following estimate holds for some constants $B>A>0$,
	\begin{equation}\label{Le3.5}
		A\lVert f\rVert_{L^{2}(\mathbb{R}^{+})}\leq\lVert T_\beta(f)\rVert_{L^{2}(\mathbb{R}^{+})}\leq B\lVert f\rVert_{L^{2}(\mathbb{R}^{+})},
	\end{equation}
	where the lower bound $A$ depends only on $\beta$.
\end{lemma}
\par We highlight that the estimate \eqref{Le3.5} was initially mentioned in \cite{MP} without providing a proof. In that reference, the inverse formulas for the operators $T_\beta$ in  $L^{2}(\mathbb{R}^{+})$  are given using hypergeometric functions. The significance of establishing the lower bound of the operators $T_\beta$ becomes evident in proving the subsequent Lemma. For the proof, we defer it to the appendix.

Next, we will prove a Logivenko-Sereda type theorem for functions $T_\beta(f)$ with supp $f\subset[a,b]$, $b>a\geq 0$.
\begin{lemma}\label{L3.6}
	$\forall$ $\beta\in \mathbb{R}$, assume  $\Omega\subset\mathbb{R}^{+}$ is a thick set, for all $f\in {L^{2}(\mathbb{R}^{+})} $ with supp $f\subset[a,b]$, $b>a\geq 0$, then there exist constant $C$ such that 
	\begin{equation}\label{Le3.6}
		C\int_{\Omega}\arrowvert T_\beta(f)(x)\arrowvert^{2}dx\geq \int_{\mathbb{R}^{+}}\arrowvert T_\beta(f)(x)\arrowvert^{2}dx,
	\end{equation}
	where $C$ depends only on $\beta$, $\Omega$ and $b-a$.
\end{lemma}
\begin{proof}
	The proof presented here is inspired by Kovrijkine's argument \cite{K}. Before proceeding our proof, we note that the constant $C$ used in the following equalities and inequalities is not fixed, it may vary from one statement to another and could depend on $\beta$, depending on whether the estimate \eqref{Le3.5} has been used.
	\par Rewrite $\cos x=(e^{ix}+e ^{-ix})/2$, then we obtain
	\begin{equation}\label{T}
		\begin{aligned}  T_\beta(f)(x)&=\sqrt{\frac{2}{\pi}}\int_{a}^{b}\cos(xy-\beta)f(y)dy\\
			&=\frac{1}{\sqrt{2\pi}}\int_{a}^{b}(e^{i(xy-\beta)}+e ^{-i(xy-\beta)})f(y)dy\\ &=\frac{1}{\sqrt{2\pi}}(e^{-i\beta}\int_{a}^{b}e^{ixy}f(y)dy+e^{i\beta}\int_{a}^{b}e^{-ixy}f(y)dy)\\ &=\frac{1}{\sqrt{2\pi}}e^{i\beta}\hat{f}(x)+\frac{1}{\sqrt{2\pi}}e^{-i\beta}\hat{f}(-x).
		\end{aligned}
	\end{equation} 
	The standard Bernstein inequality says for all functions $f$ with supp\,$\hat{f}\subset[-b,b]$, then we have the following estimate
	\begin{equation}\label{Bernstein}
		\| f^{(n)} \|_{L^{2}(\mathbb{R})} \leq  b^n \| f \|_{L^{2}(\mathbb{R})}
	\end{equation}
	for every $n\in \mathbb{N}$.
	\par Now, for every $f\in L^2(\mathbb{R}^{+})$ with $supp\,f\subset[a,b]$, we denote $g(x)=f(x+a)$, then $supp\,g(x)\subset[0,b-a]$. By \eqref{T}, we have
	\begin{equation*}
		T_\beta(f)(k)=\frac{1}{\sqrt{2\pi}}(e^{iak}e^{-i\beta}\hat{g}(-k)+e^{-iak}e^{i\beta}\hat{g}(k)).
	\end{equation*}
	Denoting $f_0(k)=T_\beta(f)(k)$, $f_1(k)=\frac{1}{\sqrt{2\pi}}e^{-i\beta}\hat{g}(-k)$, $f_2(k)=\frac{1}{\sqrt{2\pi}}e^{i\beta}\hat{g}(k)$, $\lambda_1=a$, $\lambda_2=-a$, we get
	$$f_0(k)=e^{i\lambda_1k}f_1(k)+e^{i\lambda_2k}f_2(k).$$
	We will need the Taylor formula for a analytic function $h(x)$
	\begin{equation}\label{Taylor}
		\begin{aligned}
			h(x)&=\sum_{l=0}^{m-1}\frac{h^{l}(s)}{l!}(x-s)^{l}+\frac{1}{(m-1)!}\int_{s}^{x}h^{m}(t)(x-t)^{m-1}dt\\
			&=p(x)+\frac{1}{(m-1)!}\int_{s}^{x}h^{(m)}(t)(x-t)^{m-1}dt,
		\end{aligned}
	\end{equation}
	where $p(x)$ is a polynomial of degree $m-1$. By the assumption that $\Omega$ is a thick set, that is, there exist constants $r>0$ and $L>0$ such that
	$$\arrowvert\Omega\cap [x,x+L]\arrowvert\geq rL,\quad \forall x\in \mathbb{R}^{+}.$$
	We now divide $\mathbb{R}^{+}$ into intervals $I_s=[s,s+L]$ of length $L$. Consider one of them: $I=[s,s+L]$. Then using the Taylor formula \eqref{Taylor} for analytic functions $f_1$ and $f_2$, we obtain
	\begin{equation*}
		\begin{aligned}
			f_0(k)&=\sum_{i=1}^{2}f_{i}(k)e^{i\lambda_{i}k}\\
			&=\sum_{i=1}^{2}p_{i}(k)e^{i\lambda_{i}k}+\frac{1}{(m-1)!}\sum_{i=1}^{2}e^{i\lambda_{i}k} \int_{s}^{k}f^{(m)}_{i}(t)(k-t)^{m-1}dt\\
			&=P(k)+T(k).
		\end{aligned}
	\end{equation*}
	Applying Holder's inequality, we have 
	\begin{equation}\label{ep3.11}
		\begin{aligned}
			\int_{I}\arrowvert T(k)\arrowvert^{2}dk&\leq \frac{2}{[(m-1)!]^{2}}\sum_{i=1}^{2}\int_{I}\arrowvert\int_{s}^{k}f^{(m)}_{i}(t)(k-t)^{m-1}dt\arrowvert^{2}dk\\
			&\leq \frac{2L^{2m}}{[m!]^{2}}\sum_{i=1}^{2}\int_{I}\arrowvert f^{(m)}_{i}\arrowvert^{2}dk.
		\end{aligned}
	\end{equation}
	\par We need the following auxiliary lemma (see \cite [Lemma 3]{K}) in the next step.
	\begin{lemma}\label{LP3.3}
		If $h(x)=\sum_{k=1}^{n}p_{k}(x)e^{i\lambda_{k}x},$ where $p_{k}(x)$ is a polynomial of degree $\leq m-1$ and $E\subset I$ is a measurable with $\arrowvert E\arrowvert>0$, then
		$$\|h\|_{L^{p}(I)}\leq(\frac{C\arrowvert I\arrowvert}{\arrowvert E\arrowvert})^{nm-\frac{p-1}{p}	}\cdot\|h\|_{L^{p}(E)}.$$
	\end{lemma}
	Using the above lemma and the estimate, we have
	\begin{equation*}
		\begin{aligned}
			\int_{I}\arrowvert f_0(k)\arrowvert^{2}dk&\leq 2\int_{I}\arrowvert P(k)\arrowvert^{2}dk+2\int_{I}\arrowvert T(k)\arrowvert^{2}dk\\
			&\leq (\frac{C\arrowvert I\arrowvert}{\arrowvert\Omega\cap I\arrowvert})^{4m-1}\int_{\Omega\cap I}\arrowvert P(k)\arrowvert^{2}dk+2\int_{I}\arrowvert T(k)\arrowvert^{2}dk\\
			&\leq (\frac{C}{r})^{4m-1}(2\int_{\Omega\cap I}\arrowvert f_0\arrowvert^{2}dk+2\int_{\Omega\cap I}\arrowvert T\arrowvert^{2}dk)+2\int_{I}\arrowvert T\arrowvert^{2}dk\\
			&\leq(\frac{C}{r})^{4m-1}\int_{\Omega\cap I}\arrowvert f_0\arrowvert^{2}dk+(\frac{C}{r})^{4m-1}\int_{I}\arrowvert T\arrowvert^{2}dk\\
			&\leq(\frac{C}{r})^{4m-1}\int_{\Omega\cap I}\arrowvert f_0\arrowvert^{2}dk+(\frac{C}{r})^{4m-1}\frac{2L^{2m}}{[m!]^{2}}\sum_{i=1}^{2}\int_{I}\arrowvert f^{(m)}_{i}\arrowvert^{2}dk.
		\end{aligned}
	\end{equation*}
	Summing over all intervals $I$, we have\\
	\par $\int_{\mathbb{R}^{+}}\arrowvert	T_\beta(f)\arrowvert^{2}dk$
	\begin{equation*}
		\begin{aligned}
			&\leq (\frac{C}{r})^{4m-1}\int_{\Omega}\arrowvert	T_\beta(f)\arrowvert^{2}dk+(\frac{C}{r})^{4m-1}\frac{2L^{2m}}{[m!]^{2}}\sum_{i=1}^{2}\int_{\mathbb{R}^{+}}\arrowvert f^{(m)}_{i}\arrowvert^{2}dk\\
			&\leq (\frac{C}{r})^{4m-1}\int_{\Omega}\arrowvert	T_\beta(f)\arrowvert^{2}dk+(\frac{C}{r})^{4m-1}\frac{2L^{2m}(b-a)^{2m}}{[m!]^{2}}\sum_{i=1}^{2}\int_{\mathbb{R}^{+}}\arrowvert f_{i}\arrowvert^{2}dk\\
			&\leq (\frac{C}{r})^{4m-1}\int_{\Omega}\arrowvert	T_\beta(f)\arrowvert^{2}dk+(\frac{C}{r})^{4m-1}\frac{2L^{2m}(b-a)^{2m}}{[m!]^{2}}\int_{\mathbb{R}^{+}}\arrowvert	T_\beta(f)\arrowvert^{2}dk\\
			&\leq (\frac{C}{r})^{4m-1}\int_{\Omega}\arrowvert	T_\beta(f)\arrowvert^{2}dk+(\frac{C}{r})^{4m}\frac{L^{2m}(b-a)^{2m}}{m^{2m}}\int_{\mathbb{R}^{+}}\arrowvert	T_\beta(f)\arrowvert^{2}dk.
		\end{aligned}
	\end{equation*}
	The second inequality follows from the Bernstein type inequality \eqref{Bernstein} for $f_1$ and $f_2$. The third follows from Lemma \ref{L3.5},
	$$\|f_i\|_{L^{2}(\mathbb{R}^{+},dk)}\leq C\|	T_\beta(f)\|_{L^{2}(\mathbb{R}^{+})},\,i=1,2.$$
	The last inequality is due to Stirling's formula for $m!$. Choose $m$ such that it is a positive integer and 
	$$(\frac{C}{r})^{2}\frac{L\cdot(b-a)}{m}\leq\frac{1}{2},$$
	e.g., $m=1+[2\frac{C^2}{r^{2}}L\cdot(b-a)]$, where $C>0$ depends only on $\beta$. Therefore,
	\begin{equation*}
		\int_{\mathbb{R}^{+}}\arrowvert T_\beta(f)\arrowvert^{2}dk\leq 2((\frac{C}{r})^{4(1+[2\frac{C^2}{r^{2}}L\cdot (b-a)])-1}\int_{\Omega}\arrowvert T_\beta(f)f\arrowvert^{2}dk,
	\end{equation*}
	and the constant $2(\frac{C}{r})^{4(1+[2\frac{C^2}{r^{2}}L\cdot (b-a)])-1}$ depends only on $b-a$, $\Omega$ and $\beta$. The proof of Lemma \ref{L3.6} is complete.
	
\end{proof}
\begin{corollary}\label{C3.1}
	Assume \( \Omega \subset \mathbb{R}^{+} \) is a thick set. For all \( \beta \in \mathbb{R} \), and for all \( \alpha \in \mathbb{R} \), and for all \( f \in L^{2}(\mathbb{R}^{+}) \) with \( \text{supp}\, f \subset [a,b] \), \( b > a \geq 0 \), there exists a constant \( C \) that depends only on \( \beta \), \( \Omega \), and \( b-a \) such that 
	\begin{equation}\label{Ce3.1}
		C \int_{\Omega} \left\arrowvert T_\beta(f)(x) \right\arrowvert^2 \, dx \geq \int_{\Omega} \left\arrowvert T_\alpha(f)(x) \right\arrowvert^2 \, dx.
	\end{equation}
\end{corollary}

\begin{proof}
	Formula \eqref{Ce3.1} is directly derived from the conclusions of Lemma \ref{L3.5} and Lemma \ref{L3.6}, since the upper bound of \( \| T_\alpha(f) \|_{L^2(\Omega)}^2 \) is controlled by \( C \| f \|_{L^2(\mathbb{R}^{+})}^2 \) with a universal constant \( C \).
\end{proof}

\begin{corollary}\label{C3.2}
	Assume \( \Omega \subset \mathbb{R}^{+} \) is a thick set. For all \( \beta \in \mathbb{R} \), and for any \( b > a \geq 0 \), and for any \( f \in L^2(\mathbb{R}^+) \) with \( \text{supp}\, f \subset [a,b] \), there exists a constant \( C \) that depends only on \( \beta \), \( \Omega \), and \( b-a \) such that
	\begin{equation}\label{Ce3.2}
		C \int_{\Omega} \left\arrowvert \int_{a}^{b} \cos(xy - \beta) f(y) \, dy \right\arrowvert^2 \, dx \geq \int_{\Omega} \left\arrowvert \int_{0}^{b-a} \cos(xy - \beta) f(y+a) \, dy \right\arrowvert^2 \, dx.
	\end{equation}
\end{corollary}

\begin{proof}
	It suffices to prove that the term on the right-hand side of the inequality has an upper bound controlled by \( C \| f \|_{L^2(\mathbb{R}^{+})}^2 \) with a universal constant \( C \). To see this, we rewrite \( \cos x = \frac{e^{ix} + e^{-ix}}{2} \), yielding
	\begin{equation*}
		\begin{aligned}
			\int_{0}^{b-a}\cos(xy - \beta)f(y+a)dy &= \int_{a}^{b}\cos(x(y-a)-\beta) f(y) dy\\
			&= \frac{1}{2} e^{i(\beta + ax)} (\hat{f}(x) + e^{-2iax} e^{-2i\beta} \hat{f}(-x)).
		\end{aligned}
	\end{equation*}
	Thus, applying the triangle inequality gives:
	\[
	\int_{\Omega} \left\arrowvert \int_{0}^{b-a} \cos(xy - \beta) f(y+a) \, dy \right\arrowvert^2 \, dx \leq 2 \| f \|_{L^2(\mathbb{R}^{+})}^2.
	\]
	This combines inequalities \eqref{Le3.6} and \eqref{Le3.5} to obtain \eqref{Ce3.2}.
\end{proof}
With the previous preparatory lemmas established, we now proceed to prove Lemma \ref{KeyL2}.
\par \textbf{Proof of Lemma \ref{KeyL2}}\\
Before proceeding with the proof, we note that the constant denoted by \( C(\Omega, \nu, b-a) \) in the following equalities and inequalities is not fixed; it may vary from one statement to another unless explicitly stated otherwise.
\par Given a thick set \( \Omega \), let \( \Omega^{\prime} = \Omega \setminus [0, b-a] \). For every function \( f \in L^2(\mathbb{R}^+) \) with \( \text{supp} \, F_\nu f \subset [a,b] \), we define the function \( \tilde{f} \) by the identity
\[
F_{\nu}(\tilde{f})(x) = F_{\nu}(f)(x+a).
\]
It can be observed that \( \text{supp} \, F_{\nu}(\tilde{f}) \subset [0, b-a] \). 
Now, for every \( \nu \), recalling the basic fact
\[
\| F_{\nu}(f) \|_{L^2(\mathbb{R}^+)} = \| f \|_{L^2(\mathbb{R}^+)}, \quad \forall f \in L^2(\mathbb{R}^+),
\]
we have
\begin{equation} \label{Pr3.1}
	\| f \|_{L^2(\mathbb{R}^+)}^2 = \| F_{\nu}(f) \|_{L^2(\mathbb{R}^+)}^2 = \| F_{\nu}(\tilde{f}) \|_{L^2(\mathbb{R}^+)}^2 = \| F_{\frac{1}{2}} F_{\nu}(\tilde{f}) \|_{L^2(\mathbb{R}^+)}^2.
\end{equation}
Since \( \text{supp} \, F_{\nu}(\tilde{f}) \subset [0, b-a] \), utilizing Lemma \ref{KeyL1} for \( F_{\frac{1}{2}} \) and Technical Lemma \ref{Technical1.2} with \( c = b-a \), we deduce that there exists a constant \( C(\Omega, b-a) \) such that
\begin{equation} \label{Pr3.2}
	\| F_{\frac{1}{2}} F_{\nu}(\tilde{f}) \|_{L^2(\mathbb{R}^+)}^2 \leq C(\Omega, b-a) \| F_{\frac{1}{2}} F_{\nu}(\tilde{f}) \|_{L^2(\Omega^{\prime})}^2.
\end{equation}
From the formulas \eqref{Bessel} and \eqref{Sec2.1.6}, we know that
\[
F_{\frac{1}{2}}(f) =\sqrt{\frac{2}{\pi}} \int_0^\infty \cos \left( x y - \frac{\pi}{2} \right) f(y)dy.
\]
Thus, we obtain
\[
\| F_{\frac{1}{2}} F_{\nu}(\tilde{f}) \|_{L^2(\Omega^{\prime})}^2 = \int_{\Omega^{\prime}} \left\arrowvert\sqrt{\frac{2}{\pi}} \int_0^{b-a} \cos \left( x y - \frac{\pi}{2} \right) F_{\nu}(f)(y + a)dy \right\arrowvert^2dx.
\]
Now, using the estimate \eqref{Ce3.2} in Corollary \ref{C3.2} with \( \beta = \frac{\pi}{2} \), we obtain
\[
\begin{aligned}
	\| F_{\frac{1}{2}} F_{\nu}(\tilde{f}) \|_{L^2(\Omega^{\prime})}^2 & = \int_{\Omega^{\prime}} \left\arrowvert \sqrt{\frac{2}{\pi}}\int_0^{b-a} \cos \left( x y - \frac{\pi}{2} \right) F_{\nu}(f)(y + a) \, dy \right\arrowvert^2dx \\
	& \leq C(\Omega, b-a) \int_{\Omega^{\prime}} \left\arrowvert\sqrt{\frac{2}{\pi}} \int_a^b \cos \left( x y - \frac{\pi}{2} \right) F_{\nu}(f)(y)dy \right\arrowvert^2 \, dx \\
	& = C(\Omega, b-a) \int_{\Omega^{\prime}} \left\arrowvert T_{-\frac{\pi}{2}} F_{\nu}(f) \right\arrowvert^2dx.
\end{aligned}
\]
This combines with the estimate \eqref{Ce3.1} in Corollary \ref{C3.1}, yielding
\[
\begin{aligned}
	\| F_{\frac{1}{2}} F_{\nu}(\tilde{f}) \|_{L^2(\Omega^{\prime})}^2 & = \int_{\Omega^{\prime}} \left\arrowvert \sqrt{\frac{2}{\pi}}\int_0^{b-a} \cos \left( x y - \frac{\pi}{2} \right) F_{\nu}(f)(y + a) \, dy \right\arrowvert^2dx \\
	& \leq C(\Omega, b-a) \int_{\Omega^{\prime}} \left\arrowvert \int_a^b\sqrt{\frac{2}{\pi}} \cos \left( x y - \frac{\pi}{2} \right) F_{\nu}(f)(y)dy \right\arrowvert^2dx \\
	& = C(\Omega, b-a) \int_{\Omega^{\prime}} \left\arrowvert T_{-\frac{\pi}{2}} F_{\nu}(f) \right\arrowvert^2dx \\
	& \leq C(\Omega, \nu, b-a) \int_{\Omega^{\prime}} \left\arrowvert T_{-\frac{\pi\nu}{2} - \frac{\pi}{2}} F_{\nu}(f) \right\arrowvert^2dx \\
	& = C(\Omega, \nu, b-a) \int_{\Omega^{\prime}} \left\arrowvert \int_a^b \sqrt{\frac{2}{\pi}}\cos \left( x y - \frac{\pi \nu}{2} - \frac{\pi}{2} \right) F_{\nu}(f)(y) \, dy \right\arrowvert^2dx.
\end{aligned}
\]
This combines with \eqref{Pr3.1} and \eqref{Pr3.2}, yielding that
\begin{equation} \label{Pr3.3}
	\| f \|_{L^2(\mathbb{R}^+)}^2 \leq C(\Omega, \nu, b-a) \int_{\Omega^{\prime}} \left\arrowvert \int_a^b \sqrt{\frac{2}{\pi}}\cos \left( x y - \frac{\pi \nu}{2} - \frac{\pi}{2} \right) F_{\nu}(f)(y) \, dy \right\arrowvert^2dx.
\end{equation}
Starting from now, we fix the constants \( C(\Omega, \nu, b-a) \) obtained in \eqref{Pr3.3} and denote it as \( C^* \). We now claim that for sufficiently large \( a \),
\begin{equation} \label{Pr3.5}
	\| f \|_{L^2(\mathbb{R}^+)}^2 \leq 4 C^* \| F_{\nu} F_{\nu}(f) \|_{L^2(\Omega^{\prime})}^2 = 4 C^* \| f \|_{L^2(\Omega^{\prime})}^2
\end{equation}
holds. For convenience, we denote \( g = F_{\nu}(f) \). To prove \eqref{Pr3.5}, we deduce by identity \eqref{Sec2.1} that
\begin{equation*}
	\begin{aligned}
		& C^* \int_{\Omega^{\prime}} \left\arrowvert \int_{a}^{b}\sqrt{\frac{2}{\pi}} \cos \left( x y - \frac{\pi \nu}{2} - \frac{\pi}{4} \right) g(y)dy \right\arrowvert^2dx\\
		& = C^* \int_{\Omega^{\prime}} \left\arrowvert \int_{a}^{b} \sqrt{\frac{2}{\pi}}\cos \left( x y - \frac{\pi \nu}{2} - \frac{\pi}{4} \right) g(y) + \sqrt{xy} R_{\nu}(xy) g(y) - \sqrt{xy} R_{\nu}(xy) g(y)dy \right\arrowvert^2 \, dx \\
		& \leq 2 C^* \int_{\Omega^{\prime}} \left\arrowvert \int_{a}^{b} \sqrt{xy} J_{\nu}(xy) g(y)dy \right\arrowvert^2 \, dx + 2 C^* \int_{\Omega^{\prime}} \left\arrowvert \int_{a}^{b} \sqrt{xy} R_{\nu}(xy) g(y) \, dy \right\arrowvert^2dx,
	\end{aligned}
\end{equation*}
where the last inequality follows from the Mean Value inequality. Combining this with \eqref{Pr3.3}, we obtain
\begin{equation} \label{Pr3.6}
	\| f \|_{L^2(\mathbb{R}^+)}^2 \leq 2 C^* \int_{\Omega^{\prime}} \arrowvert \int_{a}^{b} \sqrt{xy} J_{\nu}(xy) g(y) \, dy \arrowvert^2 \, dx + 2 C^* \int_{\Omega^{\prime}}\arrowvert \int_{a}^{b} \sqrt{xy} R_{\nu}(xy) g(y) \, dy \arrowvert^2 \, dx.
\end{equation}
When \( a \geq \frac{1}{b-a} \), by Hölder's inequality and estimate \eqref{Sec2.2}, we get that
\begin{equation*}
	\begin{aligned}
		& 2 C^* \int_{\Omega^{\prime}} \left\arrowvert \int_{a}^{b} \sqrt{xy} R_{\nu}(xy) g(y)dy \right\arrowvert^2dx \\
		& \leq 2 C^* \int_{\Omega^{\prime}} \int_{a}^{b} \left\arrowvert \sqrt{xy} R_{\nu}(xy) \right\arrowvert^2 dy \int_{a}^{b} \left\arrowvert g(y) \right\arrowvert^2dydx \\
		& \leq 2 C^* \int_{b-a}^{\infty} \int_{a}^{b} \left\arrowvert \sqrt{xy} R_{\nu}(xy) \right\arrowvert^2dy \int_{a}^{b} \left\arrowvert g(y) \right\arrowvert^2 dydx \\
		& \leq 2 C^* C_{\nu} \int_{b-a}^{\infty} \int_{a}^{b} \frac{1}{x^2 y^2} \, dy \int_{a}^{b} \left\arrowvert g(y) \right\arrowvert^2 dy dx \\
		& = 2 C^* C_{\nu} \frac{1}{b-a} \left( \frac{1}{a} - \frac{1}{b} \right) \| g \|_{L^2}^2 \leq 2 C^* C_{\nu} \frac{1}{a^2} \| g \|_{L^2}^2 \\
		& = 2 C^* C_{\nu} \frac{1}{a^2} \| f \|_{L^2(\mathbb{R}^+)}^2,
	\end{aligned}
\end{equation*}
where \( C_\nu \) is the constant that appears in estimate \eqref{Sec2.2}. Then, if 
\[
a \geq \max \left\{ 2 \sqrt{C^* C_{\nu}}, \frac{1}{b-a} \right\},
\]
we have
\begin{equation*}
	2 C^* \int_{\Omega^{\prime}} \left\arrowvert \int_{a}^{b} \sqrt{xy} R_{\nu}(xy) g(y) \, dy \right\arrowvert^2 \, dx \leq \frac{1}{2} \arrowvert f \|_{L^2}^2.
\end{equation*}
Combining this with \eqref{Pr3.6}, we get the desired estimate \eqref{Pr3.5}. Since \( \Omega^{\prime} \subset \Omega \) by our definition of \( \Omega^{\prime} \), we can conclude that for every
\[
a \geq \max \left\{ 2 \sqrt{C^* C_{\nu}}, \frac{1}{b-a} \right\},
\]
we obtain
\begin{equation} \label{Pr3.8}
	\| f \|_{L^2(\mathbb{R}^+)}^2 \leq 4 C^* \| f \|_{L^2(\Omega)}^2.
\end{equation}
When 
\begin{equation*} 
	a<\max \left\{ 2 \sqrt{C^* C_{\nu}}, \frac{1}{b-a} \right\},
\end{equation*}
it follows that
\begin{equation*} 
	[a,b] \subset \left[ 0, \max \left\{ 2 \sqrt{C^* C_{\nu}}, \frac{1}{b-a} \right\} + b-a \right] . 
\end{equation*}
Since the constant \( \max \left\{ 2 \sqrt{C^* C_{\nu}}, \frac{1}{b-a} \right\} \) depends only on \( \Omega \), \( \nu \), and \( b-a \), by taking \[
h = \max \left\{ 2 \sqrt{C^* C_{\nu}}, \frac{1}{b-a} \right\} + b-a
\] in Lemma \ref{KeyL1}, there exists a constant $C_1(\Omega, \nu, b-a)$ such that
\begin{equation} \label{Pr3.9}
	\| f \|_{L^2(\mathbb{R}^+)}^2 \leq C_1(\Omega, \nu, b-a) \| f \|_{L^2(\Omega)}^2.
\end{equation}
Thus, for a given thick set \( \Omega \), by combining \eqref{Pr3.8} and \eqref{Pr3.9}, if we let
\[
C = \max \left\{ 4 C^*, C_1(\Omega, \nu, b-a) \right\},
\]
which depends only on \( \Omega \), \( \nu \), and \( b-a \), we obtain the spectral inequality
\[
\| f \|_{L^2(\mathbb{R}^+)}^2 \leq C \| f \|_{L^2(\Omega)}^2,
\]
as required by Lemma \ref{KeyL2}. Hence, the proof of Lemma \ref{KeyL2} is complete.$\hfill{\Box}$

\begin{appendix}
	\section{Proof of Lemma \ref{L3.2}}
	We will now prove Lemma \ref{L3.2}. For this purpose, we recall Lemma \ref{L3.2} as follows.	
	\begin{lemma}\label{A}
		Let $\Omega\subset \mathbb{R}^{+}$ be a measurable set, $\nu\geq0$. Then the following statements are equivalent: \\
		\\
		(i)There exist constant $r>0$ and $L>0$ such that $\Omega$ is $(r,L)$-$u_{\nu}$ thick. \\
		\\
		(ii)There exist constant $r_1>0$ and $L_1>0$ such that $\Omega$ is $(r_1,L_1)$-thick.
		\\
		\\where $(r,L)$ and $(r_1,L_1)$ depend only on each other and $\nu$.
	\end{lemma}
	\begin{proof}
		We divide the argument into two steps.\\
		\\
		\textbf{Step 1} We show that (i) of Lemma \ref{A} implies (ii) of Lemma \ref{A}.\\
		\par For every $\nu\geq0$, suppose $\Omega\subset \mathbb{R}^{+}$ is $u_{\nu}$-thick, which means that there exist $r > 0$ and $L > 0$ such that
		\begin{equation*}
			\int_{\Omega\cap [x,x+L]}t^{2\nu+1}dt\geq r\int_{[x,x+L]}t^{2\nu+1}dt,\,\forall x\in \mathbb{R}^{+}.
		\end{equation*}
		For convenience, let's denote $\kappa:=2\nu+1$. Since $\nu \geq 0$, it follows that $\kappa \geq 1$. 
		Using this notation, we can rewrite the inequality as follows:
		\begin{equation*}
			\int_{\Omega\cap [x,x+L]}t^{\kappa}dt\geq r\int_{[x,x+L]}t^{\kappa}dt,
			\ \kappa\geq 1.
		\end{equation*}
		By using the fact that $t^{\kappa}$ is monotonically increasing on $[0,\infty)$, we obtain
		\begin{equation*} 
			\int_{\Omega\cap [x,x+L]}(x+L)^{\kappa}dt\geq \int_{\Omega\cap [x,x+L]}t^{\kappa}dt\geq \int_{\Omega\cap [x,x+L]}t^{\kappa}dt=\frac{(x+L)^{\kappa+1}-x^{\kappa+1}}{\kappa+1}r,
		\end{equation*}
		which yields
		\begin{equation*}
			\arrowvert\Omega\cap [x,x+L]\arrowvert=\int_{\Omega\cap [x,x+L]}dt \geq \frac{r}{\kappa+1}\frac{(x+L)^{\kappa+1}-x^{\kappa+1}}{(x+L)^{\kappa}}.
		\end{equation*}
		Now, by using the fact that $$\frac{x^{\kappa+1}}{(x+L)^{\kappa}}\leq \frac{x^{\kappa+1}}{x^{\kappa}}, $$
		we obtain that
		\begin{equation*}
			\arrowvert\Omega\cap [x,x+L]\arrowvert=\int_{\Omega\cap [x,x+L]}dt \geq \frac{r}{\kappa+1}\frac{(x+L)^{\kappa+1}-x^{\kappa+1}}{(x+L)^{\kappa}}\geq\frac{rL}{\kappa+1}=\frac{r}{2\nu+2}\arrowvert[x,x+L]\arrowvert
		\end{equation*}
		for all $x\in \mathbb{R}^{+}$.
		Thus we obtain that $\Omega$ is a thick set with constant $\left(\frac{r}{2\nu+2}, L\right)$ if it is $(r,L)$-$\mu_{\nu}$ thick.\\
		\\
		\textbf{Step 2} We show that (ii) of Lemma \ref{A} implies (i) of Lemma \ref{A}.\\
		\par
		\par Assume that $\Omega$ is thick, meaning that there exist $r$ and $L$ such that 
		\begin{equation*}
			\arrowvert \Omega\cap [x,x+L]\arrowvert\geq rL.
		\end{equation*}
		We deduce that $\Omega$ is also $u_{\nu}$-thick, that is, there exist constants $r_1$ and $L_1$ such that
		\begin{equation*}
			\mu_{\nu}(\Omega\cap [x,x+L_1])=\int_{\Omega\cap [x,x+L_1]}t^{\kappa}dt\geq r_1\int_{[x,x+L_1]}t^{\kappa}dt=r_1\mu_{\nu}([x,x+L_1]),\quad \forall x\in \mathbb{R}^{+},
		\end{equation*}
		where $\kappa:=2\nu+1$ as noted before. To this end, we take the parameter $L_1$ as $L$. Next, we proceed to show there exists a constant $r_1>0$ such that
		\begin{equation}\label{A1.1}
			\int_{\Omega\cap [x,x+L]}t^{\kappa}dt\geq r_1\int_{[x,x+L]}t^{\kappa}dt,\quad \forall x\in \mathbb{R}^{+}.
		\end{equation}
		Since $t^{\kappa}$ is monotonically increasing on $[0,\infty)$, by using technical Lemma \ref{Technical1.1} with $A=\Omega\cap [x,x+L]$, we obtain
		\begin{equation*}
			\int_{\Omega\cap [x,x+L]}t^{\kappa}dt\geq\int_{[x,x+rL]}t^{\kappa}dt.
		\end{equation*}
		Thus, to establish \eqref{A1.1}, it suffices to show that there exists a constant $r_1>0$ such that inequality
		\begin{equation*}
			\int_{[x,x+rL]}t^{\kappa}dt\geq r_1\int_{[x,x+L]}t^{\kappa}dt
		\end{equation*}
		holds for all $x\in [0,\infty)$. This reduces to estimating the minimum value of the function
		\begin{equation}\label{A1.2}
			f(x):=\frac{(x+rL)^{\kappa+1}-x^{\kappa+1}}{(x+L)^{\kappa+1}-x^{\kappa+1}}
		\end{equation}
		on $[0,\infty)$. We now proceed with the analysis by considering several cases.
		\\
		\\\textbf{Case 1}. $x>L$. Using the Mean Value theorem and the monotonicity of $t^{\kappa}$, we have
		$$\displaystyle{\frac{(x+rL)^{\kappa+1}-x^{\kappa+1}}{(x+L)^{\kappa+1}-x^{\kappa+1}}}=\frac{(\kappa+1)(x+\theta_{1}(x)rL)^{\kappa}rL}{(\kappa+1)(x+\theta_{2}(x)L)^{\kappa}L}
		\geq(\frac{x}{x+L})^{\kappa}r\geq(\frac{x}{2x})^{\kappa}r=(\frac{1}{2})^{\kappa}r,$$
		where $\theta_{1}(x), \theta_{2}(x)\in (0,1).$\\
		\textbf{Case 2}. $L\geq x>rL$. Again, using the Mean Value theorem and the monotonicity of $t^{\kappa}$, we have
		\begin{equation*}
			\displaystyle{\frac{(x+rL)^{\kappa+1}-x^{\kappa+1}}{(x+L)^{\kappa+1}-x^{\kappa+1}}}=\frac{(\kappa+1)(x+\theta_{1}rL)^{\kappa}rL}{(\kappa+1)(x+\theta_{2}L)^{\kappa}L}=(\frac{1+\theta_{1}\frac{rL}{x}}{1+\theta_{2}\frac{L}{x}})^{\kappa}r\geq\frac{(\frac{rL}{x})^{\kappa}}{(1+\frac{L}{x})^{\kappa}}r\geq\frac{r^{\kappa+1}}{2^{\kappa}}.
		\end{equation*}
		\textbf{Case 3}. $rL\geq x\geq 0.$\\
		Since taking the derivative of the function is quite complex, we provide a rough estimate of the function's lower bound in this case here. Due to the continuity of the function \eqref{A1.2} and the fact that $f(0)=r^{k+1}$, So there exists $\varepsilon>0$ depending only on $\kappa$, $r$ and $L$, such that
		\begin{equation*}
			\frac{(x+rL)^{\kappa+1}-x^{\kappa+1}}{(x+L)^{\kappa+1}-x^{\kappa+1}}>\frac{1}{2}r^{\kappa+1}
		\end{equation*}
		holds for all $x\in[0,\varepsilon]$. For $\varepsilon\leq x \leq rL$, we have
		$$\frac{(x+rL)^{\kappa+1}-x^{\kappa+1}}{(x+L)^{\kappa+1}-x^{\kappa+1}}\geq \frac{2^{\kappa+1}x^{\kappa+1}-x^{\kappa+1}}{(x+L)^{\kappa+1}}=\frac{2^{\kappa+1}-1}{(1+\frac{L}{x})^{\kappa+1}}\geq \frac{2^{\kappa+1}-1}{(1+\frac{L}{\varepsilon})^{\kappa+1}}.$$
		Finally, by taking 
		$$r_1=\min\left\{(\frac{1}{2})^{2\nu+1}r, (\frac{1}{2})^{2\nu+1}r^{2\nu+2}, \frac{1}{2}r^{2\nu+1+1}, \frac{2^{2\nu+1+1}-1}{(1+\frac{L}{\varepsilon})^{2\nu+2}}\right\}>0,$$
		we obtain that
		\begin{equation*}
			\int_{\Omega\cap [x,x+L]}t^{2\nu+1}dt\geq\int_{[x,x+rL]}t^{2\nu+1}dt\geq r_1\int_{[x,x+L]}t^{2\nu+1}dt.
		\end{equation*}
		Thus, we conclude that $\Omega$ is a $\mu_\nu$ thick set with constant $(r_1,L)$, depending only on $r$, $L$ and $\nu$, which end the proof of Lemma \ref{L3.2}.
	\end{proof}
	\section{Proof of Lemma \ref{L3.5}}
	In this appendix, we provide a proof for Lemma \ref{L3.5}. For this purpose, we recall Lemma \ref{L3.5} as follows.
	\begin{lemma}\label{B}
		For every $\beta\in \mathbb{R}$, the operators $T_\beta$ defined by \eqref{TB} extends to bounded invertible maps on $L^2(\mathbb{R}^{+})$. Moreover the following estimate holds for some constants $B>A>0$,
		\begin{equation}\label{B1.1}
			A\lVert f\rVert_{L^{2}(\mathbb{R}^{+})}\leq\lVert T_\beta(f)\rVert_{L^{2}(\mathbb{R}^{+})}\leq B\lVert f\rVert_{L^{2}(\mathbb{R}^{+})},
		\end{equation}
		where the lower bound $A$ depends only on $\beta$.
	\end{lemma}
	\begin{proof}
		By the definition of operators $T_\beta$, for $f\in L^2(\mathbb{R}^{+})$, we have
		\begin{equation*}
			\begin{aligned}
				T_\beta(f)&=\sqrt{\frac{2}{\pi}}\cos(\beta)\int_{0}^{\infty}\cos(xy)f(y)dy+\sqrt{\frac{2}{\pi}}\sin(\beta)\int_{0}^{\infty}\sin(xy)f(y)dy\\
				&=\cos(\beta)F_c(f)+\sin(\beta)F_s(f),
			\end{aligned}
		\end{equation*}
		where $F_c$ and $F_s$ denote the Fourier-Cosine and Fourier-Sine transform, respectively, and both are unitary operators in $L^2(\mathbb{R}^{+})$. Using a standard trigonometric inequality, we obtain
		\begin{equation*}
			\| \cos(\beta) - \arrowvert\sin(\beta)\arrowvert \| \cdot \| f \|_{L^{2}(\mathbb{R}^{+})} \leq \| T_\beta(f) \|_{L^{2}(\mathbb{R}^{+})} \leq (\cos(\beta) + \arrowvert\sin(\beta)\arrowvert) \| f \|_{L^{2}(\mathbb{R}^{+})}.
		\end{equation*}
		Thus the upper bound \eqref{B1.1} is trivial, and we can set $B=2$ for any parameter $\beta$. To establish the lower bound \( A > 0 \) in \eqref{B1.1}, it is sufficient to verify its validity for \( \beta = \pm \frac{\pi}{4} \). For $\beta=\frac{\pi}{4}$, the operator $T_{\frac{\pi}{4}}$ corresponds to the half-Hartley transform, and Theorem 1 in \cite{S} provides known lower and upper bounds for the half-Hartley transform:
		$$\sqrt{2}\Arrowvert f\Arrowvert_{L^{2}(\mathbb{R}^{+})}\leq\Arrowvert T_{\frac{\pi}{4}}(f)\Arrowvert_{L^{2}(\mathbb{R}^{+})}\leq2\Arrowvert f\Arrowvert_{L^{2}(\mathbb{R}^{+})}.$$
		For $\beta=-\frac{\pi}{4}$, we can show that
		\begin{equation}\label{B1.2}
			T_{-\frac{\pi}{4}}=H_1\circ T_{\frac{\pi}{4}}^{-},
		\end{equation}
		where $H_1$ (as defined in \eqref{B1.3} below) is the one-sided Hilbert transform, and $T_{\frac{\pi}{4}}^{-}$ is the inverse of $T_{\frac{\pi}{4}}$. Since the one-sided Hilbert transform has a bounded inverse operator on $L^2(\mathbb{R}^{+})$ (see \cite [Chapter 12]{FWK}), we obtain \eqref{B1.1} for $T_{-\frac{\pi}{4}}$.
		To prove \eqref{B1.2}, it suffices to show that for every $f\in C^{\infty}_{0}(\mathbb{R}^{+})$,
		\begin{equation}\label{B1.3}
			T_{-\frac{\pi}{4}}\circ T_{\frac{\pi}{4}}(f)=H_1(f)=\frac{1}{\pi}PV\int_0^{\infty}\frac{f(y)}{y-x}dy.
		\end{equation}
		Indeed, 
		$$T_{-\frac{\pi}{4}}\circ T_{\frac{\pi}{4}}(f)(x)=\frac{2}{\pi}\int_{0}^{\infty}\cos(xz+\frac{\pi}{4})dz\int_{0}^{\infty}\cos(zy-\frac{\pi}{4})f(y)dy.$$
		It can be computed by the formal change of order of integration if we write
		$$\frac{2}{\pi}\int_{0}^{\infty}\cos(xz+\frac{\pi}{4})dz\int_{0}^{\infty}\cos(zy-\frac{\pi}{4})f(y)dy$$
		$$=\frac{2}{\pi}\lim_{\varepsilon\downarrow 0}\int_{0}^{\infty}f(y)dy\int_{0}^{\infty}e^{-\varepsilon z}\cos(xz+\frac{\pi}{4})\cos(zy-\frac{\pi}{4})dz.$$
		We  easily find that
		\begin{equation*}
			\begin{aligned}
				\int_{0}^{\infty}\cos(xz+\frac{\pi}{4})\cos(zy-\frac{\pi}{4})dz&=\frac{1}{4}(\int_{-\infty}^{0}e^{-z(x+y)i}dz+\int_{0}^{\infty}e^{-z(x+y)i}dz)\\
				&+\frac{1}{4}(e^{\frac{\pi}{2}i}\int_{-\infty}^{0}e^{-z(x-y)i}dz+e^{-\frac{\pi}{2}i}\int_{0}^{\infty}e^{-z(x-y)i}dz).
			\end{aligned}
		\end{equation*}
		Then recalling the Fourier transform of homogeneous distributions $\chi^{\alpha}_{\pm}$ (Cf. \cite [p.167]{LH})
		$$F(\chi^{a}_{\pm})(\xi)=e^{\mp i\pi(a+1)/2}(\xi\mp i0)^{-a-1},$$
		we obtain 
		$$\int_{0}^{\infty}\cos(xz+\frac{\pi}{4})\cos(zy-\frac{\pi}{4})dz=\frac{\pi}{2}\delta(x+y)-\frac{1}{2}PV\frac{1}{x-y}.$$
		When restricted to $x\in \mathbb{R}^{+}$ and $f\in L^{2}(\mathbb{R}^{+})$, we get \eqref{B1.2}. This completes the proof.
	\end{proof}
\end{appendix}
\section*{Acknowledgments}
Z. Duan was supported by the National Natural Science Foundation of China under grants 61671009 and 12171178. Longben Wei thanks Han Cheng and Professor Shanlin Huang for some useful discussion.
\section*{Declarations}
\textbf{Conflict of interest} The authors declare that there is no conflict of interest.



\vspace{2em}

	
\end{document}